\documentclass{amsart}

\usepackage{enumitem,amsrefs,xcolor,amsthm}

\newtheorem{theorem}{Theorem}[section]
\newtheorem{lemma}[theorem]{Lemma}
\newtheorem{proposition}[theorem]{Proposition}
\newtheorem{corollary}[theorem]{Corollary}
\newtheorem{conjecture}[theorem]{Conjecture}

\theoremstyle{definition}

\newtheorem{example}[theorem]{Example}

\numberwithin{equation}{section}

\newcommand{\R}{\mathbb{R}}

\newcommand{\N}{\mathbb{N}}

\newcommand{\hA}{\mathcal{A}}
\newcommand{\hB}{\mathcal{B}}

\newcommand{\hH}{\mathcal{H}}

\newcommand{\sbsts}[2]{\binom{[#1]}{#2}}

\newcommand{\ins}{\mathrm{ins}}
\DeclareMathOperator{\myspan}{span}
\DeclareMathOperator{\GL}{GL}

\DeclareMathOperator{\rank}{rank}

\makeatletter
\newcommand{\expow}{\@ifnextchar^\@extp{\@extp^{\,}}}
\def\@extp^#1{\mathop{\bigwedge\nolimits^{\!#1}}}
\makeatother

\newcommand{\innprod}{\,\llcorner\,}
\newcommand{\pown}{\mathcal{P}(n)}
\newcommand{\Ffull}{F_{\textrm{full}}}

\newcommand{\inversions}[2]{\rho(#1,#2)}

\definecolor{fixed}{rgb}{0,0,0}
\definecolor{more}{rgb}{0,0,0}

\begin{document}

\title[Combinatorics in the exterior algebra]{Combinatorics in the exterior algebra and the Bollob\'as Two Families Theorem}

\author{Alex Scott}
\address{Mathematical Institute, University of Oxford, Andrew Wiles Building, Radcliffe Observatory Quarter, Woodstock Road, Oxford, OX2 6GG, UK}
\email{scott@maths.ox.ac.uk}
\thanks{Alex Scott was supported by a Leverhulme Trust Research Fellowship.}

\author{Elizabeth Wilmer}
\address{Department of Mathematics, Oberlin College, Oberlin, OH, 44074, USA}
\email{ewilmer@oberlin.edu}
\thanks{Elizabeth Wilmer gratefully acknowledges the hospitality of Merton College
and the Mathematical Institute of the University of Oxford while this work
was being completed.}


\date{\today}

\begin{abstract}
We investigate the combinatorial structure of subspaces of the exterior algebra of a finite-dimensional real vector space, working in parallel with the extremal combinatorics of hypergraphs. \textcolor{black}{Using initial monomials, projections of the underlying vector space onto subspaces, and the interior product, we find analogs of local and global LYM inequalities, the Erd\H{o}s-Ko-Rado theorem, and the Ahlswede-Khachatrian bound for $t$-intersecting hypergraphs.}

Using these tools, we prove a new extension of the Two Families Theorem of Bollob\'{a}s\textcolor{black}{, giving a weighted bound for subspace configurations satisfying a skew cross-intersection condition. We also verify a recent conjecture of Gerbner, Keszegh, Methuku, Abhishek, Nagy, Patk\'{o}s,
Tompkins, and Xiao on pairs of set systems satisfying both an intersection and a cross-intersection condition.}
\end{abstract}

\maketitle

%
%

\section{Introduction} 
\label{sec:intro}

For several decades there have been useful links between exterior algebra and combinatorics.  Constructions exploiting the wedge product have been used in combinatorics to study intersections in hypergraphs, saturation problems, and simplicial complexes; the exterior algebra approach~\cites{Lovasz77,F82,Alon85,AK85,K84} to Bollob\'{a}s's celebrated  Two Families Theorem~\cite{Bollobas65} is a highlight, as is Kalai's method of algebraic shifting~\cites{KalaiAlgShift,Kalai84a,Kalai86}.  Conversely, combinatorial results can be used to elucidate algebraic structures: a central example is the Kruskal-Katona theorem~\cites{Kruskal,Katona,Schutz}, which characterizes $f$-vectors of simplicial complexes and Hilbert series in certain algebraic structures.   

In this paper, we study the combinatorics of linear subspaces of the exterior algebra of a finite dimensional real vector space.  We prove new results both in the exterior algebra and in extremal set theory.  As an application of our results, we prove a new extension of the Two Families Theorem of Bollob\'as. \textcolor{black}{We also affirmatively resolve a recent conjecture of  Gerbner, Keszegh, Methuku, Abhishek, Nagy, Patk\'{o}s,
Tompkins, and Xiao~\cite{GKMNPTX}.}

The paper is organized as follows.
In Section~\ref{sec:setup}, we first recall the basic correspondence between uniform hypergraphs and homogeneous subspaces of the exterior algebra over $\R^n$ (which depends on both a basis for $\R^n$ and a term ordering of the corresponding monomials in~$\expow^r \R^n$).
We then use the correspondence to prove results about subspaces of the exterior algebra, \textcolor{black}{developing subspace analogues for several intersection conditions on hypergraphs.}
For example, we determine the maximum dimension of a subspace of $V=\expow^r\R^n$ in which every pair of elements has wedge product 0,
and the maximum of $(\dim U)(\dim W)$ over subspaces $U$ of $\expow^r V$ and $W$ of $\expow^s V$ that mutually annihilate. We also give exterior analogs for local LYM inequalities and use them to prove a global LYM inequality for graded ideals in the exterior algebra.

Section~\ref{sec:projections} considers projections and liftings in the exterior algebra.
We prove dimensional fraction bounds for projections and liftings of homogeneous subspaces of the exterior algebra (Lemmas~\ref{prop:extproj} and~\ref{prop:increaseboth}). In fact the exterior algebra setting allows us more freedom than the combinatorial setting, since a generic choice of basis ensures that images under ``random'' projections have constant dimension (Corollaries~\ref{cor:genericproj} and~\ref{cor:genericboth}). In Section~\ref{sec:twofam}, we prove Theorem \ref{thm:subspaces} and Corollary \ref{cor:combo}, which are new extensions of the Bollob\'{a}s Two Families Theorem for both subspaces and set systems. The proof relies on both the exterior local LYM inequality and our bounds on generic projections. 

\textcolor{black}{In Section~\ref{sec:addapp} we show that the size of pairs of families satisfying both the Two Family hypotheses and an intersection condition on the first family is bounded by the Ahlswede-Khachatrian bound on the size of $t$-intersecting families, as conjectured by Gerbner, Keszegh, Methuku, Abhishek, Nagy, Patk\'{o}s,
Tompkins, and Xiao~\cite{GKMNPTX}. Finally, in Section~\ref{sec:badexamples} we collect some limiting examples and propose a few questions.}

We work over the reals throughout, although our arguments would go through over the complex numbers, or any field of characteristic 0.
%
%

\section{Exterior algebra and hypergraphs}\label{sec:setup}

 \begin{quote}
It happens to be rather easy to express the size of an $r$-graph in terms of exterior powers, but to make use of this expression is a rather different matter. \hfill \cite{WhiteBook}*{p.~117}
\end{quote}
\medskip

We begin this section by setting up definitions and notation, and defining the connection between hypergraphs and subspaces of the exterior algebra.  We then use this connection to prove results about self-annihilating subspaces and pairs of mutually annihilating subspaces of the exterior algebra, and on the change in dimensional fraction when a subspace is wedged with the underlying space or contracted with the dual space.

\subsection{Monomial subspaces and initial hypergraphs}\label{sec:exterior}

Given an integer $n >0$, we write $[n] = \{1, \dots, n\}$. For $0 \leq r \leq n$,  we write
\[
\sbsts{n}{r} = \{A \subseteq [n]\,:\, |A| = r \}
\]
for the collection of $r$-element subsets of $[n]$ \textcolor{black}{and $\pown = \bigcup_{r=0}^n \sbsts{n}{r}$ for the collection of all subsets of $[n]$.   A \textit{hypergraph $\mathcal A$ with ground set $[n]$} is a subset of $\pown$.} We call $\mathcal{A}$ {\em $r$-uniform}
when $\mathcal A\subseteq \sbsts{n}{r}$. 

For exterior algebra we largely follow the notation and terminology of~\cite{Bourbaki}, \cite{Bourbaki89}, \cite{FultonHarris}*{Appendix B3}, \cite{AHH97}, and~\cite{HH11}*{Chapter 5}, but we emphasize the dependence on a basis. The results included in Section~\ref{sec:setup} do not depend on the basis; however, some results in Section~\ref{sec:projections} will require a generic basis. 

Let $V=\R^n$, viewed as column vectors, and write 
\[
\expow V = \bigoplus_{r=0}^n \expow^r V 
\]
for the standard grading of the exterior algebra of $V$. \textcolor{black}{We also write $\expow^{<t} V = \bigoplus_{r=0}^{t-1} \expow^r V$}.

\textcolor{black}{We call $v \in \expow^rV$ an \textit{$r$-vector} and say $v$ has \textit{degree $r$}. When there exist $v_1,\dots, v_r \in V$ such that $v = v_1 \wedge \dots \wedge v_r$, we call $v$ \textit{decomposable}, or an \textit{$r$-blade}.}

Let $E=\{e_1,\dots,e_n\}$ be the standard basis for $V$. 
For an arbitrary element $F \in \GL_n(\R)$, we denote the columns and the entries of the $E$-matrix for $F$ by
\[
F = (f_1 | \dots |f_n) = (f_{ij}).
\] 
Often we will identity $F$ with the ordered basis $\{f_1, \dots, f_n\}$ formed by the columns of its standard matrix. For $A \in \sbsts{n}{r}$, write $f_A = \bigwedge_{a \in A} f_a \in \expow^r V$, where the elements of $A$ are listed in increasing order. 
For $A,B \subseteq [n]$, we have \textcolor{fixed}{
\begin{equation}\label{eq:extmonmult}
f_A \wedge f_B = \begin{cases}
0 & A \cap B \not= \emptyset,\\
(-1)^{\inversions{A}{B}} f_{A \cup B} & A \cap B = \emptyset,
\end{cases}
\end{equation}
where we define
\begin{equation}\label{eq:inversions}
\inversions{A}{B} = 
|\{(a,b) \in A \times B\,:\, a>b\}|
\end{equation}
to be  the number of inversions between disjoint sets $A,B \in \N$. (The resulting sign is the same as the sign of the permutation sorting the concatenation of the sorted listings of $A$ and $B$.)}

The set $F_r=\left\{f_A\, : \, A \in \sbsts{n}{r}\right\}$ is a basis for $\expow^r V$ and $\dim \expow^r V = \binom{n}{r}$. \textcolor{black}{We  write $\Ffull = \bigcup_{r=0}^n F_r$, so that $\Ffull$ is a basis for $\expow V$, and  $\dim \expow V = 2^n$.}  For a hypergraph $\hA \subseteq \textcolor{black}{\pown}$, write $F(\hA)=\myspan\{f_A\,:\, A \in \hA \}$. Note that $\dim F(\hA) = |\hA|$ and that $f_A$ and $F(\hA)$ both depend on our choice of $F$.

We call a subspace $W \subseteq \expow V$ \textit{monomial with respect to $F$} when $W= F(\hA)$ for some hypergraph $\hA \subseteq \textcolor{black}{\pown}$. \textcolor{black}{Note that $\hA \mapsto F(\hA)$ forms a bijection between  hypergraphs with ground set $[n]$} and subspaces of $\expow V$ monomial with respect to the fixed basis $F$; see Lemma~\ref{lem:inithyp}.  

\textcolor{black}{Given a non-zero $w \in \expow V$, define its 
\textit{initial set $\ins_F(w) \in \pown$ with respect to $F$}  as follows: expand $w$ in the basis $\Ffull$ as $w = \sum_{A \in \pown} m_A f_A$. Let 
\[
\ins(w) = \max\left\{A\in\pown\,:\, m_A \not= 0\right\}.
\]
where the maximum is taken with respect to the following ordering of $\pown$: first, sort from largest cardinality to smallest. Then use \textit{reverse colex order} within $\sbsts{n}{r}$. More formally, for $A,B \in \sbsts{n}{r}$, we say $A>B$ exactly when $|A|>|B|$, or $|A|=|B|$ and $\max(A \Delta B) \in B$.
For example, $\ins(f_1 \wedge f_4 \wedge f_5 + f_2 \wedge f_3 \wedge f_5+ f_1 \wedge f_2) = \{2,3,5\}$.} See, for example,~\cite{WhiteBook}*{Chapter 5} or \cite{Anderson2002}*{Chapter 7} for combinatorial treatments of colex order.  The corresponding ordering on monomials is sometimes called \textit{reverse lex} in the algebraic combinatorics literature, see for example~\cite{HH11}*{Section 2.1.2}.

The key property of our ordering of \textcolor{black}{$\pown$ is that it is a \textit{term order}, that is,} 
\begin{equation}\label{eq:colex}
A > B \textrm{ if and only if } A \cup C > B \cup C, \textrm{ whenever } A \cap C = B \cap C = \emptyset.
\end{equation}
It follows immediately that for $C \subseteq [n]$ and $w \in \expow V$ satisfying $\ins(w) \cap C = \emptyset$, 
\begin{equation}\label{eq:initialset}
\ins (w \wedge f_C) = \ins(w) \cup C. 
\end{equation}

We define the \textit{initial hypergraph $\hH_F(W) \subseteq \textcolor{black}{\pown}$ with respect to $F$} of a subspace $W \subseteq \expow V$ by
\[
\hH_F(W) = \{ \ins(w)\,:\, w\in W, w\not=0\}.
\]

Let us note some basic facts about the correspondence between hypergraphs and subspaces.

\begin{lemma}\label{lem:inithyp} 
Let $V=\R^n$ and $F \in \GL_n(\R)$. Then
\begin{enumerate}[label={\textup{(\roman*)}}] 
\item \label{it:dim} $\dim W = |\hH_F(W)|$ for any subspace $W \subseteq \expow V$.
\item \label{it:monomialinverse} $F(\hH_F(W))=W$ for $W$ monomial with respect to $F$.
\item \label{it:inverse} $\hH_F(F(\hA)) = \hA$ for any $\hA \subseteq \textcolor{black}{\pown}$. 
\end{enumerate}
\end{lemma}
\begin{proof} For~\ref{it:dim}, note that the elements of any basis of $W$ whose matrix in $\Ffull$ is in reduced row echelon form \textcolor{black}{with respect to our ordering of $\pown$} must have distinct initial sets. That $\Ffull$ is a basis of $\expow V$ implies ~\ref{it:inverse}, and~\ref{it:monomialinverse} follows by applying $F$ to both sides (i.e., sending the hypergraphs to the correpsonding $F$-monomial subspaces). 
\end{proof}

We note that taking initial monomials, often with respect to a generic basis, is an important tool in the study of monomial ideals (see e.g.~\cite{HH11}); generally it is applied to ideals, but we will be interested almost everywhere in mere subspaces (Theorem~\ref{thm:LYM} is the only exception). It is also easy to describe Kalai's algebraic shifting~\cite{KalaiAlgShift} in this notation: the algebraic shift of a hypergraph $\hA$ with ground set $[n]$ is the hypergraph $\hH_F(I(\hA))$, where the identity matrix $I$ induces the standard basis of $\R^n$, and $F \in \GL_n(\R)$ is generic. We will use genericity in a similar spirit, but will need to be able to modify the dimension of the underlying vector spaces; see Sections~\ref{sec:genericproj} and~\ref{sec:dimfracs}. 

%
%

\subsection{\textcolor{black}{Intersection and Annihilation}}

We define a hypergraph $\hA \subseteq 2^{[n]}$ to be \textit{intersecting} if $A\cap B \not= \emptyset$ for all  $A,B \in  \hA$.  It is easy to see that if $\hA$ is intersecting then $|\hA|\le 2^{n-1}$, as $\hA$ can contain at most one set from each pair $\{A,[n]\setminus A\}$. For $r>n/2$, it is clear that any $r$-uniform hypergraph is intersecting.  However, for $r\le n/2$, the situation is more interesting.   The classical Erd\H os-Ko-Rado Theorem, which is both an important tool in extremal combinatorics and the center of a web of generalizations (see, \textcolor{black}{for instance,} Godsil and Meager~\cite{GM16}), gives an optimal bound on the size of an $r$-uniform intersecting family.

\begin{theorem}[Erd\H{o}s, Ko, Rado \cite{EKR}] \label{thm:EKR}
Let $\hA \subseteq 2^{[n]}$ be an intersecting hypergraph. Then $|\hA| \leq 2^{n-1}$.  
Furthermore,  if $\hA $ is $r$-uniform, where $r\le n/2$, then
\[
|\hA| \leq \binom{n-1}{r-1}.
\]
\end{theorem}

What is the appropriate exterior analogue? Define a subspace $W \subseteq \expow V$ to be \textit{self-annihilating} if $v\wedge w=0$ for all $v,w\in W$. This definition allows a direct extension of Theorem~\ref{thm:EKR} to subspaces of $\expow V$.

\begin{theorem}
\label{prop:extEKRspaces}
Let $V=\R^n$ and let $W$ be a self-annihilating subspace of $\expow V$. Then
$\dim W\le 2^{n-1}.$
Furthermore,  if $W \subseteq \expow^r V$, where $r\le n/2$, then
\begin{equation}\label{selfie}
\dim W \leq \binom{n-1}{r-1}.
\end{equation}
\end{theorem}

Our proof of Theorem~\ref{prop:extEKRspaces} uses the correspondence between hypergraphs and subspaces developed in Section~\ref{sec:exterior}; we show that the initial hypergraph of a self-annihilating space must be intersecting.  We note that Woodroofe~\cite{Woodroofe} has recently given an alternative proof of our Theorem~\ref{prop:extEKRspaces}, based on the Borel Fixed Point Theorem for the actions of algebraic groups on projective varieties.

\begin{proof}[Proof of Theorem \ref{prop:extEKRspaces}] Fix $F \in \GL_n(\R)$. By the Erd\H{o}s-Ko-Rado Theorem and Lemma~\ref{lem:inithyp}, it is enough to verify that $\hH_F(W)$ is an intersecting \textcolor{black}{hypergraph}, as
$\dim(W)=|\hH_F(W)|$. 
Assume, looking for a contradiction, that for some nonzero $u, w \in W $ we have $A \cap B = \emptyset$, where $A = \ins_F(u)$ and $B=\ins_F(w)$. Since $u \wedge w =0$, there must be other sets $A',B'$ in the supports of $u,w$ respectively with $A' \cap B' = \emptyset$ and $A' \cup B' = A \cup B$ (or else $f_{A \cup B}$ will have non-zero coefficient when we expand $u\wedge w$ in the $F$-monomial basis $\Ffull$). \textcolor{black}{It must be true that $|A'|=|A|$ and $|B'|=|B|$, since $A$ and $B$ are both initial sets (so $|A| \geq |A'|$ and $|B| \geq |B'|$) and $|A'| +|B'| = |A | + | B|$.}

Let $A_0 = A \cap A' $,  $B_0=B \cap B'$, $X = A \cap B'$, and $Y= B \cap A'$.  
This gives disjoint decompositions
\begin{align*}
 A & = A_0 \cup X, \quad &  B  &= B_0 \cup Y, &\\
 A' & = A_0 \cup Y, \quad &  B' &=  B_0 \cup X,&
\end{align*}
so by~\eqref{eq:colex} 
\[
A > A' \iff X >Y \iff B'>B,
\]
contradicting either $A=\ins_F(u)$ or $B=\ins_F(w)$. 
\end{proof}

Both parts of Theorem~\ref{prop:extEKRspaces} are optimal.  For any fixed vector $v\in V$, the space
$\{v\wedge z: z\in \expow V\}$ has dimension $2^{n-1}$. For $r\le n/2$ the space 
$\{v\wedge z: z\in \expow^{r-1} V\}$ has dimension $\binom {n-1}{r-1}$.  

For $r<n/2$, the extremal cases in Theorem \ref{thm:EKR} have a nice characterization: there is a single element contained in all sets of the family. It is an interesting question to describe the extremal examples for Theorem~\ref{prop:extEKRspaces}. This is trivially true for $r=1$; it is also true for $r=2$, and follows from the fact that in this case
elements of self-annihilating spaces are decomposable. Could all all extremal examples be of this form? 

%
%


\subsection{\textcolor{black}{Mutually annihilating pairs of subspaces}}

We now consider pairs of subspaces.  Two subspaces $U,W$ of the exterior algebra are {\em mutually annihilating} if $u\wedge w=0$ for all $u\in U$ and $w\in W$.  We have the following counterpart to Theorem \ref{prop:extEKRspaces} (which implies \eqref{selfie} in the special case where we take $U=W$).

\begin{theorem}
\label{prop:crossintspaces}
Let $V=\R^n$ and $1\le r,s\le n/2$.  Suppose that $U \subseteq \expow^r V$ and $W\subseteq \expow^s V$, and $u\wedge w=0$ whenever $u\in U$ and $w\in W$.  Then
\[
\dim U \dim W \leq \binom{n-1}{r-1}\binom{n-1}{s-1}.
\]
\end{theorem}

\begin{proof}
This follows similar lines to the proof of Theorem \ref{prop:extEKRspaces}: we consider the hypergraphs $\hA=\hH_F(U)$ and $\hB=\hH_F(W)$.  Then $\hA$ is $r$-uniform, $\hB$ is $s$-uniform, and (arguing as before) we have $A\cap B$ nonempty for all $A\in\hA$ and $B\in \hB$.  This means that $\hA$ and $\hB$ are {\em cross-intersecting systems}, and so by results of Pyber \cite{P86} and Matsumoto and Tokushige \cite{MT89} we have
$$(\dim U)(\dim W)=|\hA| |\hB|\le \binom{n-1}{r-1}\binom{n-1}{s-1},$$
as required.
\end{proof}

Note that it is possible to attain equality in Theorem \ref{prop:crossintspaces} by fixing $v\in V$ and setting 
$U= \{v\wedge z: z\in\expow^{r-1}\textcolor{black}{V}\}$ and $W=\{v\wedge z: z\in\expow^{s-1}\textcolor{black}{V}\}$.  As with Theorem \ref{prop:extEKRspaces}, it would be interesting to characterize the extremal examples when $r,s<n/2$.

%
%

\subsection{(Upwards) Local and global LYM for the exterior algebra} 
\label{sec:LYMup}

The \textit{LYM inequality} of
Lubell, Meshalkin and Yamamoto \cites{Lubell,Meshalkin,Yamamoto} is a central result in extremal set theory.

\begin{theorem}[LYM inequality]\label{thm:realLYM} Let $\hA = \{A-i\dots, A_m\} \subseteq \pown$ be an antichain under the containment order (that is, $A_i \not\subseteq A_j$ for all $i \not= j$). Then 
\[
\sum_{i=1}^m \frac{1}{\binom{n}{|A_i|}} \leq 1. 
\]
\end{theorem}
One approach to proving the LYM inequality relies on elementary counting bounds known as \textit{Local LYM inequalities}, a version of which can be found as far back as Sperner \cite{Sp28}. Let $\hA \subseteq \sbsts{n}{a}$ be an $a$-uniform hypergraph.  \textcolor{black}{For $1\le c\le n-a$, the \textit{$c$-th upper shadow} of $\hA$ is the hypergraph
\[
\partial^c \hA = \left\{B \in \sbsts{n}{a+c}: B \supseteq A \textrm{ for some } A \in \hA \right\}.
\]
For $1\le c\le a$, the \textit{$c$-th lower shadow} of $\hA$ is the hypergraph
\[
\partial_c \hA = \left\{C \in \sbsts{n}{a-c}: B \subseteq A \textrm{ for some } A \in \hA \right\}.
\]}

\begin{lemma}[Local LYM inequality]\label{lem:localLYM}
Let $\hA \subseteq \sbsts{n}{a}$ be an $a$-uniform hypergraph with ground set $[n]$. For any $ 0\leq c \leq n-a$, 
\textcolor{black}{\begin{align}
\frac{|\partial^b \hA|}{\binom{n}{a+c}} & \geq \frac{|\hA|}{\binom{n}{a}}\label{eq:LYMup}\\
\intertext{For any $0 \leq c \leq a$,}
\frac{|\partial_c \hA|}{\binom{n}{a-c}} & \geq \frac{|\hA|}{\binom{n}{a}}.\label{eq:LYMdown}
\end{align}}
For both directions, equality holds if and only if $\hA = \emptyset$ or $\hA = \sbsts{n}{a}$. 
\end{lemma} 
Note that set complementation interchanges the upwards direction \eqref{eq:LYMup} and the downwards direction \eqref{eq:LYMdown} of Lemma~\ref{lem:localLYM}. The Local LYM Inequality is also known as the \textit{normalized matching property}. Kleitman~\cite{Kleitman} proved that for finite ranked posets the normalized matching property is equivalent to the LYM bound 
on the size of an antichain.  

Both Theorem~\ref{thm:realLYM} and Lemma~\ref{lem:localLYM} carry over to exterior algebra. Theorem~\ref{thm:LYM}, a version of the full LYM inequality,  bounds minimal generating sets for graded ideals in $\expow V$. The proof of Theorem~\ref{thm:LYM} from an exterior upwards Local LYM (Theorem~\ref{prop:extLYM}) parallels a standard inductive proof of the full LYM inequality from the upwards Local LYM inequality (see, e.g.,~\cite{WhiteBook}*{p.\ 13}). Since the downwards exterior Local LYM (Theorem~\ref{prop:extLYMdown}) will require the additional machinery of interior products to state and prove, we postpone it to section~\ref{sec:downLYM}.

For subspaces $U, W \subseteq \expow V$, define
$$U \wedge W = \myspan\{u \wedge w\,:\, u\in U,\,w \in W \};$$
we also write $U \wedge w = U \wedge \myspan\{w\}$.

Fix $F \in \GL_n(\R)$. For a monomial subspace $F(\hA) \subseteq \expow^r V$,    equations~\eqref{eq:extmonmult} and~\eqref{eq:colex} imply
\begin{equation}\label{bercy2}
F(\hA) \wedge \expow^c V = \myspan\left\{ f_A \wedge f_J \, : \, A \in \hA,\, J \in \sbsts{n}{c}\right\} = F(\partial^c \hA).
\end{equation}
That is, for monomial spaces, wedging with an exterior power of the ground space yields the monomial space generated by the upper shadow of the initial hypergraph.  
It follows \textcolor{black}{from Lemma~\ref{lem:localLYM}} that 
$$\frac{\dim\left(F(\hA) \wedge \expow^c V \right)}{\binom{n}{r+c}} \geq 
\frac{\dim(F(\hA))}{\binom{n}{r}}.$$
Note that the denominators satisfy $\binom{n}{r+c} = \dim(\expow^{r+c}V)$ and $\binom{n}{r} = \dim(\expow^r V)$, respectively, so we have bounded the dimensional fraction of the ambient space $\expow^{r+c}V$ occupied by the wedge product space $F(\hA) \wedge \expow^c V$. 

\textcolor{black}{What about general homogeneous subspaces? Let $W \subseteq \expow^rV$. If $A \in \hH_F(W)$, then there exists $w \in W$ with $\ins(w)=A$. For $ A\cap B = \emptyset$, 
equations \eqref{eq:extmonmult} and~\eqref{eq:colex} imply that $\ins(w \wedge B) = A\cup B$. Clearly $w \wedge f_B \in W \wedge\expow^{|B|} V$. Hence we have the containment}
\begin{equation}\label{eq:extshadow}
\hH_F\left(W \wedge \expow^c V\right) 
\supseteq \left\{ A\cup B\! : \, A \in \hH_F(W), B \in \binom{[n]\setminus A}{c} \right\}
=\partial^c (\hH_F(W)), 
\end{equation}
and this suffices to prove a Local LYM bound. 
\begin{theorem}[Upwards Local LYM in the exterior algebra]\label{prop:extLYM} Let $V=\R^n$  and $W \subseteq \expow^r V$. Then
for $0 \leq c \leq n-r$, 
\[
\frac{\dim\left( W\wedge  \expow^c V\right)}{\binom{n}{r+c}} \geq \frac{\dim W}{\binom{n}{r}}.
\]
Equality occurs only when $W= \{0\}$ or $W = \expow^r V$. 
\end{theorem}

\begin{proof} 
Fix $F \in \GL_n(\R)$. By~\eqref{eq:extshadow} and Lemmas~\ref{lem:localLYM} and ~\ref{lem:inithyp},
\begin{equation*}
\frac{\dim (W \wedge \expow^c V)}{\binom{n}{r+c}} 
=\frac{|\hH_F(W\wedge\expow^c V)|}{\binom{n}{r+c}}
 \geq \frac{|\partial^c \hH_F(W)|}{\binom{n}{r+c}} \geq \frac{|\hH_F(W)|}{\binom{n}{r}}
= \frac{\dim W}{\binom{n}{r}}.
\end{equation*}
Equality implies
\[
\frac{|\partial^c \hH_F(W)|}{\binom{n}{r+c}}= \frac{|\hH_F(W)|}{\binom{n}{r}}.
\]
By Theorem~\ref{thm:realLYM}, $\hH_F(W) = \emptyset$ or $\hH_F(W) = \sbsts{n}{r}$, so by Lemma~\ref{lem:inithyp}, $W= \{0\}$ or $W = \expow^r V$.
\end{proof}

Theorem~\ref{prop:extLYM} can be viewed as a comparison of the $r$- and $(r+c)$-entries in the $f$-vector of the graded $\expow V$-ideal generated by $W$. The result could also be deduced from a suitable version of the  Kruskal-Katona theorem for $\expow V$~(as found, for example, in \cite{AHH97}*{Theorem 4.1}).

We say $I \subseteq \expow V$ is a \textit{graded ideal} when it is an ideal in $\expow V$ and
\[
I = \bigoplus_{i=0}^n \left(I \cap \expow^r V\right). 
\]
It is equivalent to require that $I$ be generated by homogeneous (although not necessarily decomposable) elements of $\expow V$. Graded ideals in $\expow V$ are two-sided. 

\begin{theorem}[Exterior LYM] \label{thm:LYM}
Let $V = \R^n$ be an $n$-dimensional real vector space, and let $I \subseteq \expow V$ be a graded ideal. Let $A = \{a_1,\dots, a_m\}$ be a minimal set of homogeneous generators for $I$, where $a_i \in \expow^{r_i} V$. Then
\[
\sum_{i = 1}^m \frac{1}{\binom{n}{r_i}} \leq 1. 
\]
Equality occurs only when  $\myspan{A} = \expow^r V$ for some $r$. 
\end{theorem}
\begin{proof} Without loss of generality, we may assume that $r_1 \leq r_2 \leq \dots \leq r_m$. Note that for each $r$, the elements $\{a_i\,|\, r_i=r\} \subseteq \expow^r V$ are linearly independent by minimality of $A$. Now define linear subspaces $Z_i \subseteq \expow^{r_i} V$ recursively by
\[
Z_1 = \myspan\{a_1\} \qquad \textrm{and} \qquad Z_{i+1}  = \myspan\left\{ \left(\expow^{r_{i+1} -r_i} V \right) \wedge Z_i,\, a_{i+1}\right\}.
\]
First, we claim that $a_{i+1} \not\in \left(\expow^{r_{i+1}-r_i} V\right) \wedge Z_i$. Why? The elements of $\left(\expow^{r_{i+1}-r_i} V \right)\wedge Z_i$ are  of the form
\[
\sum_{j=1}^i w_j \wedge a_j,
\]
where $w_j \in \expow^{r_{i+1}-r_j} V$, and if $a_{i+1}$ were of this form, then $A$ would not be a minimal generating set of the ideal $I$.
Theorem~\ref{prop:extLYM} implies that for each $i$
\begin{equation}\label{eq:LYMstep}
\frac{\dim \left(\expow^{r_{i+1}-r_i} V \wedge Z_i \right)}{\binom{n}{r_{i+1}}} \geq \frac{\dim Z_i}{\binom{n}{r_i}}
\end{equation}
and thus
\[
\frac{\dim Z_{i+1}}{\binom{n}{r_{i+1}}} = \frac{1+\dim \left(\expow^{r_{i+1}-r_i} V \wedge Z_i \right)}{\binom{n}{r_{i+1}}} \geq \frac{1}{\binom{n}{r_{i+1}}} + \frac{\dim Z_i}{\binom{n}{r}}.
\]
We can now proceed recursively down from $1 \geq \frac{\dim Z_m}{\binom{n}{r_m}}$.

If equality occurs, then we must have equality in~\eqref{eq:LYMstep} for each $i$. By Theorem~\ref{prop:extLYM}, that is only possible when $r_{i+1}=r_i$ or $Z_i = \expow^{r_i} V$. Hence $\{a_1,\dots, a_m\}$ forms a basis for some $\expow^r V$. 
\end{proof}

%
%

\subsection{Interior products and a downwards Local LYM}
\label{sec:downLYM}

In order to state an exterior downwards Local LYM inequality, we will need to use the \textit{interior product} $\innprod:\, \expow V \times \expow V^* \rightarrow \expow V$, where $V$ is an $n$-dimensional real vector space and $V^*$ is its dual. Interior products can reduce exterior grade, which is a necessary ingredient for a downwards Local LYM. Earlier applications of the interior product to combinatorics include Kalai's exterior matroids of hypergraphs~\cites{Kalai85,Kalai90}, which were also studied by Pikhurko~\cite{Pikhurko}, and  Karasev's exterior algebra presentation~\cite{Karasev} of Huang's spectacular proof of the Sensitivity Conjecture~\cite{Huang}.

For duality in the exterior algebra and interior products, we largely follow the notation of Fulton and Harris~\cite{FultonHarris}*{Appendix B3} and Bourbaki~\cite{Bourbaki89}*{Chapter III, \S 11}.  
The vector space dual of $\expow V$ can be identified with $\expow V^*$. Indeed, for each $k$,  $\left(\expow^k V\right)^*$ can be identified with $\expow^k V^*$, and 
\begin{equation}\label{eq:dualgrade}
\langle v_1 \wedge \dots \wedge v_k, w^*_1 \wedge \dots \wedge w^*_k \rangle = \det (w^*_j(v_i))
\end{equation}
for any $v_1,\dots, v_k \in V$, $w^*_1, \dots, w^*_k \in V^*$ \textcolor{more}{(see~\cite{FultonHarris}*{p.\ 476})}. Then taking the direct sum of the duals in each grade gives the dual of the entire graded space. 

It will be useful to work with duals in an explicit basis. Let $F=\{f_1,\dots,f_n\}$  be a basis for $V$. We write $F^*=\{f^*_1,\dots, f^*_n\}$ for the corresponding dual basis for the dual space $V^*$, for which the dual pairing $\langle \cdot ,\cdot \rangle$ satisfies
\[
\langle f_j, f^*_i \rangle = f^*_i(f_j) = \begin{cases}
1 & i=j,\\
0 & \textrm{otherwise}.
\end{cases}
\]Consistent with our earlier notation, define $f^*_A = f^*_{a_1} \wedge \dots \wedge f^*_{a_k} \in \expow^k V^*$ for $A = \{a_1, \dots, a_k\} \subseteq [n]$, where we have listed the elements $a_1, \dots, a_k$ of $A$ in increasing order.
Then $\{f^*_A\,:\, A \subseteq [n]\}$ is a basis for $\expow V^*$, and, for $A,B \subseteq [n]$, by \eqref{eq:dualgrade} and the grading structure we have
\[
\langle f_B, f^*_A  \rangle = f^*_A(f_B) = \begin{cases}
1 & A=B,\\
0 & \textrm{otherwise}.
\end{cases}
\]

\textcolor{black}{
We can now define  the interior product. For $0 \leq b \leq a \leq n$, we define $\innprod:\, \expow V \times \expow V^* \rightarrow \expow V$ to be the transpose of the wedge product in $\expow V^*$. That is, for any $v \in \expow V$,  $w^*, u^* \in \expow V^*$, 
\begin{equation}\label{eq:intprodpairing}
\langle v \innprod w^*, u^* \rangle = \langle v, w^* \wedge u^* \rangle.
\end{equation}
(Recall that the angle brackets denote the duality pairing). 
Let $\{f_1,\dots, f_n\}$  be a basis of $V$ and let  
$\{f^*_1,\dots, f^*_n\}$ be the corresponding dual basis of $V^*$. Then \eqref{eq:intprodpairing} and~\eqref{eq:extmonmult} imply that for $A, B \subseteq [n]$, \textcolor{fixed}{
\begin{equation}\label{eq:intproddef}
f_A \innprod f^*_B = 
 \begin{cases}
 (-1)^{\inversions{B}{A\setminus B}} f_{A \setminus B} & B \subseteq A,\\
0 & \textrm{otherwise}.
\end{cases}
\end{equation}
where $\inversions{B}{A\setminus B}$ is defined in~\eqref{eq:inversions}.} Notice that when $ 0 \leq b \leq a \leq n$ and $x \in \expow^a V$, $y^* \in \expow^b V^*$, we have $x\innprod\,y^* \in \expow^{a-b} V$. We will sometimes call an interior product a \textit{contraction}.}

\textcolor{black}{ 
For subspaces $U \subseteq \expow V$ and $W^* \subseteq \expow V^*$, define
\[
U \innprod W^* = \myspan\{ u\innprod w^* \,:\, u\in U,\,w^* \in W^* \}.
\]
For $w^* \in W^*$, we also write $ U \innprod w^* =  U \innprod  \myspan\{w^*\} = \myspan\{u\innprod w^* \,:\, u \in U\} $.}

\textcolor{black}{We are ready to build our downwards Local LYM. Fix a basis $F$ for $V$. For a monomial subspace $F(\hA) \subseteq \expow^r V$,    equations~\eqref{eq:intproddef} and~\eqref{eq:colex}
imply 
\begin{equation}\label{bercy}
F(\hA)   \innprod \expow^c V^*   = \myspan\left\{ f_A  \innprod  f^*_J \, : \, A \in \hA,\, J \in \sbsts{n}{c}\right\} = F(\partial_c \hA).
\end{equation}
That is, for monomial spaces, contracting with an exterior power of the dual of the ground space yields the monomial space generated by the corresponding lower shadow of the initial hypergraph.  
It follows from Lemma~\ref{lem:localLYM} that 
$$\frac{\dim\left(  F(\hA) \innprod \expow^c V^*   \right)}{\binom{n}{r-c}} \geq 
\frac{\dim(F(\hA))}{\binom{n}{r}}.$$}

\textcolor{black}{What about general homogeneous subspaces? Let $W \subseteq \expow^rV$. If $A \in \hH_F(W)$, then there exists $w \in W$ with $\ins(w)=A$. For $B \subseteq A$, 
equations \eqref{eq:intproddef} and~\eqref{eq:colex} imply that $\ins(w \innprod f^*_B) = A\setminus B$. Clearly $w \innprod f^*_B \in W \innprod \expow^{|B|} V$. Hence we have the containment
\begin{equation}\label{eq:extshadowdown}
\hH_F\left(  W \innprod \expow^c V^* \right) 
\supseteq \left\{ A\setminus B\! : \, A \in \hH_F(W), B \in \binom{A}{c} 
\right\}
=\partial_c (\hH_F(W)).
\end{equation}
}
\begin{theorem}[Downwards Local LYM in the exterior algebra]\label{prop:extLYMdown} \textcolor{black}{Let $V=\R^n$  and $W \subseteq \expow^r V$. Then
for $0 \leq c \leq r$, 
\[
\frac{\dim\left( W \innprod \expow^c V^*\right) }{\binom{n}{r-c}} \geq \frac{\dim W}{\binom{n}{r}}.
\]}
Equality occurs only when $W= \{0\}$ or $W = \expow^r V$. 
\end{theorem}

\begin{proof} 
\textcolor{black}{Fix $F \in \GL_n(\R)$. By~\eqref{eq:extshadowdown} and Lemmas~\ref{lem:localLYM} and ~\ref{lem:inithyp},
\begin{equation*}
\frac{\dim \left(W  \innprod \expow^c V^* \right) }{\binom{n}{r-c}} 
=\frac{|\hH_F\left( W \innprod \expow^c V^* \right) |}{\binom{n}{r-c}}
 \geq \frac{|\partial^c \hH_F(W)|}{\binom{n}{r-c}} \geq \frac{|\hH_F(W)|}{\binom{n}{r}}
= \frac{\dim W}{\binom{n}{r}}.
\end{equation*}}
Equality implies
\[
\frac{|\partial^c \hH_F(W)|}{\binom{n}{r-c}} = \frac{|\hH_F(W)|}{\binom{n}{r}}.
\]
By Theorem~\ref{thm:realLYM}, $\hH_F(W) = \emptyset$ or $\hH_F(W) = \sbsts{n}{r}$, so by Lemma~\ref{lem:inithyp}, $W= \{0\}$ or $W = \expow^r V$.
\end{proof}


%
%

\subsection{\textcolor{black}{$t$-self-annihilation via interior products}}

\textcolor{black}{For $t>0$,  a hypergraph $\hA$ is called $t$-\textit{intersecting} when $|A \cap B| \geq t$ for all $A,B \in \mathcal{A}$.  Just as we did for intersecting hypergraphs, we would like to define an analogous notion in the exterior algebra: we will call these subspaces $t$-\textit{self-annihilating}. We will use interior products to do so, and then generalize Theorem~\ref{prop:extEKRspaces} to $t \geq 1$.}

\textcolor{black}{To motivate the upcoming definition of $t$-self-annihilating, we note that a hypergraph $\hA$ is $t$-intersecting exactly when the hypergraph $\{ A \setminus C\,:\, A \in \hA\}$ is intersecting for all sets $C$ having at most $t-1$ elements. In parallel, we define a subspace $W \subseteq \expow V$ to be \textit{$t$-self-annihilating} when 
\begin{equation}\label{eq:tselfanndef}
(u \innprod\, y^* ) \wedge ( w \innprod\,y^* ) = 0 
\end{equation}
for all $u,v \in W$ and all \textcolor{fixed}{decomposable} $y^* \in \expow ^{<t} V^*$. Note that $1$-self-annihilating coincides with self-annihilating as defined above, since $\expow^0 V^*$ is a copy of the field $\R$ of scalars: $f^*_\emptyset = 1$.} Also note that \eqref{eq:tselfanndef} implies that when $W$ is $t$-self-annihilating, then the space $W \innprod\, y^*$ is self-annihilating for every decomposable $y^* \in \expow ^{<t} V^*$.

\textcolor{black}{It is immediate from the definition of $t$-intersecting that every edge of a $t$-intersecting hypergraph must have  cardinality at least $t$. Proposition~\ref{prop:tsadegreebig} verifies a parallel property for $t$-self-annihilating spaces. }
\begin{proposition}\label{prop:tsadegreebig} \textcolor{black}{ Let $V$ be an $n$-dimensional real vector space and fix $t>0$.
When a subspace $W \subseteq \expow V$ is $t$-self-annihilating, then $W \subseteq \expow^{\geq t} V$.}
\end{proposition}
\begin{proof} \textcolor{black}{Assume not. Fix a basis $F$ for $V$ and fix $w$ such that $w \in W$, but $w \not\in \expow^{\geq t} V$. Write 
$w = \sum_{A \subseteq [n]} m_A f_A$.
Then there exists $B \subseteq [n]$ with $r=|B|<t$ and $m_B \not= 0$. It follows that $f^*_B  \in \expow^{<t} V^*$. \textcolor{fixed}{Note that $f^*_B$ is decomposable, and that~\eqref{eq:intproddef} implies that
\begin{align}\label{eq:fixer1}
w \innprod f^*_B & = \left(\sum_{A \in [n]} \alpha_A f_A \right)\innprod f^*_B
= \sum_{C \subseteq [n]\setminus B} ( \pm \alpha_{C \cup B} ) f_C 
\end{align}}
 By~\eqref{eq:extmonmult} and bilinearity, the $f_\emptyset$-term in $(w \innprod f^*_B) \wedge (w \innprod f^*_B)$ \textcolor{fixed}{is the product of the $C=\emptyset$ terms in~\eqref{eq:fixer1}, that is,} 
\begin{align*}
 \pm  (\alpha_B f_\emptyset) \wedge (\alpha_B f_\emptyset) & =  \pm \alpha_B^2 f_\emptyset\not= 0.
\end{align*}
This contradicts our hypothesis that $W$ is $t$-self-annihilating. }
\end{proof}

Theorem~\ref{thm:tintersectingspaces} generalizes Theorem~\ref{prop:extEKRspaces} to $t$-self-annihilating spaces. We postpone the proof, which is parallel to the proof of Theorem~\ref{prop:extEKRspaces} but uses some geometric properties of interior products, to the end of this section.  
\textcolor{black}{
\begin{theorem}\label{thm:tintersectingspaces} Let $F= \{f_1,\dots, f_n\}$ be a basis for an $n$-dimensional real vector space $V$ and fix $t>0$.
\begin{enumerate}[label=\textup{(\arabic*)}]
\item \label{it:setstospaces} If $\hA \subseteq 2^{[n]}$ is a $t$-intersecting hypergraph, then the corresponding monomial subspace $F(\hA) \subseteq \expow V$ is $t$-self-annihilating.
\item \label{it:spacestosets} If $W \subseteq \expow V$ is a $t$-self-annihilating subspace, then $\hH_F(W)$ is $t$-intersecting. 
\end{enumerate}
\end{theorem}
}

\textcolor{more}{Theorem~\ref{thm:tintersectingspaces} allows us to adapt extremal results on $t$-intersecting set systems to bound the dimension of $t$-self-annihilating subspaces of the exterior algebra. }
\textcolor{black}{
The question of the maximum size of an $a$-uniform $t$-intersecting family $\hA \subseteq \pown$ \textcolor{more}{was considered by Erd\H{o}s, Ko, and Rado~\cite{EKR}, who showed that for sufficiently large $n$ the answer is $\binom{n-t}{a-t}$. 
} The question was resolved for all $n$ by the Complete Intersection Theorem of Ahlswede and Khachatrian~\cite{AhKh}; it is standard to denote the function they found as $AK(n,a,t)$.  Theorem~\ref{thm:tintersectingspaces} immediately implies
\begin{theorem}\label{thm:AKspaces} When $V$ is an $n$-dimensional real vector space and $W$ is a $t$-self-annihilating subspace of $\expow^a V$, then $\dim(W) \leq AK(n,a,t)$. 
\end{theorem}}
\textcolor{black}{In Section 5 below we use Theorem~\ref{thm:AKspaces} to settle a conjecture of Gerbner, Keszegh, Methuku, Abhishek, Nagy, Patk\'{o}s,
Tompkins, and Xiao~\cite{GKMNPTX}.}

\textcolor{more}{Erd\H{o}s, Ko, and Rado~\cite{EKR} also raised the question of the maximum size of  an arbitrary $t$-intersecting family $\hA \subseteq \pown$ and conjectured an answer in the case that $t+n$ is even. Katona~\cite{Katona64} gave a full extremal characterization of such families; see Ahlswede and Khachatrian~\cite{AK05} for additional discussion. Katona's result and Theorem~\ref{thm:tintersectingspaces} immediately imply
\begin{theorem}\label{thm:Katonaspaces} Let $V$ be an $n$-dimensional real vector space and let $W$ be a $t$-self-annihilating subspace of $\expow V$.
\begin{itemize}
\item If $n+t$ is even, then 
\[
\dim W \leq \sum_{i = \frac{n+t}{2}}^n \binom{n}{i}.
\]
\item If $n+t$ is odd, then 
\[
\dim W \leq 2 \sum_{i = \frac{n+t-1}{2}}^n \binom{n-1}{i}.
\]
\end{itemize}
\end{theorem}}

\textcolor{more}{To start building towards the proof of Theorem~\ref{thm:tintersectingspaces}, we first record in  Lemma~\ref{lem:decomp} a standard fact about the geometry of interior products: 
 the interior product of decomposables is itself decomposable. See, e.g.~\cite{FultonHarris}*{Appendix B}.}

\begin{lemma}\label{lem:decomp} \textcolor{more}{Let $x \in \expow^t V$ and $y^* \in \expow^r V^*$ be non-zero decomposable eements, with $0 \leq r <t \leq n = \dim V$. Write 
$x=x_1\wedge\dots\wedge x_t$ and  $y^* = y^*_1 \wedge\dots \wedge y^*_r$.}

\textcolor{more}{Define subspaces $X,Z \subseteq V$ by 
\[
X = \myspan\{x_1\dots, x_t\} \quad \textrm{and} \quad  Z = \ker y^*_1 \cap \dots \cap \ker y^*_r.
\] Then $x \innprod y^* \in \expow^{t-r} V$ is decomposable. Furthermore, $x \innprod y^* \not= 0 $ exactly when $\dim (X \cap Z)=t-r$, and in this case, for any decomposition $x \innprod y^* = u_1\wedge \dots \wedge u_{t-r}$, we have $\myspan\{u_1,\dots, u_{t-r}\} = X \cap Z$. }
\end{lemma}

The next Lemma  will be useful for checking that particular subspaces of~$\expow V$ are $t$-self-annihilating.
\begin{lemma}\label{lem:tselfann} \textcolor{more}{Let $u = u_1 \wedge \dots \wedge u_p$ and  $w = w_1\wedge\dots\wedge w_q$ be decomposable elements of $\expow^p V$, $\expow^q V$, respectively, and let $U =\myspan\{u_1,\dots, u_p\}$ and $W = \myspan\{w_1,\dots,w_q\}$ be the corresponding subspaces of $V$. If $\dim (U \cap W) =t$ and $y^* = y_1\wedge \dots \wedge y_r$ is a decomposable element of $\expow^r V^*$, where $r<t$, then 
\[
(u \innprod y^*) \wedge (w \innprod y^*) = 0. 
\]}
\end{lemma}
\begin{proof} Let $Z = \ker y^*_1 \cap \dots \cap \ker y^*_r$. When $y^* \not=0$, then $\dim Z = n-r$.  By Lemma~\ref{lem:decomp}, it will suffice to show that $(U \cap Z) \cap (W \cap Z) \not=  \{0\}$, since then we can decompose $u \innprod y^*$ and $w \innprod y^*$ to each have a non-zero vector in that intersection as a wedge factor. 
However, 
\[
\dim ((U \cap W) \cap Z) +  \dim (\myspan\{U \cap W, Z\}) = \dim (U \cap W) + \dim Z = t + (n-r)
\]
implies that $\dim ((U \cap W) \cap Z) \geq t -r >0$, and $(U \cap W) \cap Z$ is a subspace of both $ U \cap Z$ and $W \cap Z$. 
\end{proof}

\begin{proof}[Proof of Theorem~\ref{thm:tintersectingspaces}]
\textcolor{black}{For~\ref{it:setstospaces},} \textcolor{more}{let $v, w \in F(\hA)$, where $\hA$ is a $t$-interesecting hypergraph, and let $y^* \in \expow^r V^*$, where $r<t$. 
Expand 
\[
v = \sum_{A \in \hA} \alpha_A f_A, \qquad w = \sum_{A \in \hA} \beta_A f_A.
\]
Then, by bilinearity,
\begin{align*}
(u \innprod y^*) \wedge (w \innprod y^*) & = \sum_{A \in \hA} \sum_{B \in \hA}
\alpha_A \beta_B\, (f_A \innprod y^*) \wedge (f_B \innprod y^*).
\end{align*}
By the $t$-intersecting property of $\hA$ and Lemma~\ref{lem:tselfann}, each term of this sum is zero. }

\textcolor{black}{The proof of~\ref{it:spacestosets} is a little more involved. Assume, looking for a contradiction, that for some nonzero $u, w \in W $, we have $|A \cap B| <t$, where $A = \ins_F(u)$ and $B=\ins_F(w)$. Set $D = A \cap B$.  
We expand both $u$ and $w$ in the basis $\Ffull$: 
\begin{equation}\label{eq:expandfull}
u = \sum_{C \subseteq [n]} \alpha_C f_C, \qquad w = \sum_{C \subseteq [n]} \beta_C f_C,
\end{equation}
and note that $\alpha_A$ and $\beta_B$ are both non-zero. Because $W$ is $t$-self-annihilating, $|D|<t$, \textcolor{more}{and $f^*_D$ is decomposable}, we know
\begin{equation}\label{eq:tself}
( u \innprod f^*_D  ) \wedge (w\innprod f^*_D ) = 
\left(\sum_{C \subseteq [n]} \alpha_C (f_C   \innprod f^*_D )\right) \wedge \left(\sum_{C \subseteq [n]} \beta_C (f_C  \innprod f^*_D) \right )= 0.
\end{equation}
The term $\alpha_A \beta_B (f_A \innprod f^*_D) \wedge (f_B \innprod f^*_D)$ of~\eqref{eq:tself} formed by the initial terms in~\eqref{eq:expandfull} is non-zero, since by~\eqref{eq:intproddef} it is a non-zero scalar multiple of $f_{A\setminus D} \wedge f_{B\setminus D} = \pm f_{(A \cup B) \setminus D}$. Hence there must exist a different pair of sets $A',B' \subseteq [n]$ such that the corresponding term $\alpha_{A'} \beta_{B'} (f_{A'} \innprod f^*_D) \wedge (f_{B'} \innprod f^*_D)$ in~\eqref{eq:tself} is also a non-zero scalar multiple of $f_{(A \cup B) \setminus D}$, which implies the following four conditions are all satisfied:
\begin{enumerate}[label={\textup{(\roman*)}}] 
\item $\alpha_{A'}\not=0$ and $\beta_{B'}\not=0$,
\item  $D \subseteq A'$ and $D \subseteq B'$ (to survive the contraction with $f^*_D$),
\item $(A' \setminus D) \cap (B'\setminus D) = \emptyset$ (to survive the wedge product), and
\item $(A' \cup B') \setminus D = (A \cup B) \setminus D$.
\end{enumerate}
Because $A$ and $B$ are the initial sets for $u$,$w$ respectively,  $|A| \geq |A'|$ and $|B| \geq |B'|$. Then (iii) and (iv) above imply that $|A'|=|A|$ and $|B'|=|B|$.}

Let $A_0 = (A  \cap A')\setminus D$,  $B_0=(B \cap B')\setminus D$, $X = (A\cap B')\setminus D$, and $Y = (B  \cap A')\setminus D$.
This gives disjoint decompositions
\begin{align*}
 A & = A_0 \cup D \cup X, \quad &  B  &= B_0 \cup D \cup Y, &\\
 A' & = A_0 \cup D \cup Y, \quad &  B' &=  B_0 \cup D \cup X,&
\end{align*}
so by~\eqref{eq:colex} 
\[
A > A' \iff X >Y \iff B'>B,
\]
contradicting either $A=\ins_F(u)$ or $B=\ins_F(w)$. 
\end{proof}

%
%

\section{Generic linear projections}
\label{sec:projections}

In this section, we will be interested in the behaviour of subspaces $W$ of $\expow^rV$ under projections and under the operation of wedging with exterior powers of $V$.  In both cases, we will want bounds on the dimension of the resulting subspace.   Note that projections change the dimension of the underlying space, while wedging with an exterior power lifts $W$ from $\expow^rV$ to a higher exterior power.

Our proofs will use suitably generic subspaces of $V$: we show the existence of such subspaces in 
section~\ref{sec:genericproj}  and prove our bounds on the dimensions of subspaces in section~\ref{sec:dimfracs}.

\subsection{Generic projections}\label{sec:genericproj}
Throughout this section, let $V= \R^N$. We find conditions that guarantee the existence of bases of $V$ that behave generically with respect to projections of given configurations of subspaces.  In all cases we find a nonempty Zariski open subset of $\GL_N(\R)$ having the desired properties (it makes no significant difference to the final results if we instead use the condition that our sets have complement with Lebesgue measure zero).

Let $F = (f_{ij})=\left(f_1| f_2 | \dots | f_N\right)\in \GL_N(\R)$;
i.e.\ $F$ is an $N \times N$ matrix with entries $f_{ij}$ and columns $f_j$.
For $J \subseteq [N]$, let $V_J = \myspan\{f_j\,:\, j \in J \}$, and define the linear projection $\pi^F_J: V \rightarrow V_J$ by
\begin{equation}\label{eq:pidef}
\pi^F_J \left( \sum_{j \in [N]} \alpha_j f_j\right) = \sum_{j \in J} \alpha_j f_j.
\end{equation}

For a subspace $C$ of $V$ and a set $J\subseteq[N]$, we clearly have
$\dim(\pi_J^F(C))\le\min\{\dim C,|J|\}$.  We will show that, for typical choices of $F$, this holds with equality.
The proof of Lemma \ref{lem:nocollapse} follows Frankl and Tokushige~\cite{FT}*{Lemma 26.14}.  

\begin{lemma}\label{lem:nocollapse} Let $C_1,\dots, C_m$ be proper linear subspaces of $V$. Then there exists a non-zero polynomial $G$ in the $N^2$ variables $f_{ij}$, $1 \leq i,j \leq N$, such that $G(F) \not= 0$ implies that $F = (f_{ij}) \in \GL_N(\R)$  and
\[
\dim \pi^F_J(C_i) = \min\{\dim C_i, |J|\}
\]
for all $1 \leq i \leq m$ and $J \subseteq [N]$.
\end{lemma}
\begin{proof} 
The key idea is to write down a polynomial witnessing that $\pi_J^F(C_i)$ has maximum possible rank. 
Let $d_i=\dim C_i$. For each $1 \leq i \leq m$ and $J \subseteq [N]$, let $M_{i,J}$ be an $N$ by $d_i+(N-|J|)$ matrix built by taking $d_i$ columns forming a basis for $C_i$, together with the $N-|J|$ columns $f_j$, where $j \in [N] \setminus J$. We choose $G_{i,J}$ to be a minor of $M_{i,J}$ that can witness $M_{i,J}$ having full rank. More precisely:
\begin{itemize}
\item If $d_i \geq |J|$, let $G_{i,J}$ be an $N \times N$ minor including all $C_i$-basis columns, together with any choice of $N-d_i \leq N-|J|$ columns $f_j$. 
\item Otherwise $d_i< |J|$. In this case there is a collection of $d_i$ rows such that the restriction of the $C_i$ basis to those rows is still linearly independent. Let $G_{i,J}$ be any $d_i + N-|J|$ by $d_i+N-|J|$ minor of $M_{i,J}$ including those $d_i$ rows (and all $d_i$ columns from the basis for $C_i$).
\end{itemize}
Note that $V_{[N]\setminus J}$ is the kernel of $\pi^F_J$. By construction, $G_{i,J} \not= 0$ implies that $\dim \myspan\{C_i,V_{[N]\setminus J}\} = \min\{N, d_i+N-|J|\}$, and thus immediately that $\dim (C_i \cap V_{[N]\setminus J}) = \max \{d_i-|J|,0\}$
and $\dim(\pi^F_J(C_i)) = \min (\dim C_i, |J|)$.

Finally, set $\displaystyle{G = (\det F) \prod_{1 \leq i \leq m,\ J \subseteq [N]} G_{i,J}}$.
We note that $G$ is not the zero polynomial, as for each $i$ and $J$ there are choices of $F$ for which the matrix $M_{i,J}$  has full rank.
\end{proof}

We need an analogous result for subspaces of $\expow^r V$.  This is more difficult than for subspaces of $V$, as the subspace structure of $\expow^rV$ interacts with the exterior algebra structure.

\begin{lemma}\label{lem:constantproj} Let $W \subseteq \expow^r V$ be a linear subspace. For $1\leq m \leq N-1$, set
\[
t_m = \max_{J,F} \dim \pi^F_J (W),
\]
where the maximum is taken over all $J \in \sbsts{N}{m}$ and $F \in \GL_N(\R)$. 
Then there exists a non-zero polynomial $H$ in the $N^2$ variables $f_{ij}$, $ 1 \leq i,j \leq N$, such that $H(F) \not= 0$ implies that $F = (f_{ij}) \in \GL_N(\R)$  and
\[
\dim \pi^F_J(W) = t_m
\]
for all $J \in \sbsts{N}{m}$. 
\end{lemma}

\begin{proof} Fix $m$. Let $d= \dim(W) \geq t_m$, and choose $J^* \in \sbsts{N}{m}$ and $F^* \in \GL_N(\R)$ to realize 
\begin{equation}\label{eq:stardim}
\dim \pi^{F^*}_{J^*} = t_m.
\end{equation}
Let $\{w_1,\dots, w_d\}$ be a basis of $W$ such that $\pi^{F^*}_{J^*}(w_1), \dots, \pi^{F^*}_{J^*}(w_{t_m})$ are linearly independent in $V_{J^*}=\pi^{F^*}_{J^*}(V)$. 

Build an $\binom{N}{r}$ by $\binom{N}{r}-\binom{m}{r}+d$  matrix $M_{J^*}$ by taking the standard coordinates of $w_1,\dots, w_d$ for the first $d$ columns, and the standard Pl\"{u}cker
 coordinates of the vectors $f_K$, where $K \in \sbsts{N}{r} \setminus \sbsts{J}{r}$, as the rest of the columns (the entries in these columns are degree-$r$ polynomials in the variables
 $f_{ij}$). For any $F \in \GL_N(\R)$, the $f_K$-columns of $M_{J^*}$ form a basis for $\ker \pi^{F}_{J^*}$. By~\eqref{eq:stardim}, for the specific basis $F^*$ we have
\[
\dim \left( W \cap \ker \pi^{F^*}_{J^*}\right) = d-t_m, 
\]
and thus, when $F=F^*$,
\[ 
\rank M_{J^*} = \binom{N}{r}-\binom{m}{r} +d  - (d-t_m)=\binom{N}{r}-\binom{m}{r} +t_m.
\]
It follows that there exists a non-zero $\binom{N}{r}-\binom{m}{r} +t_m$ by $\binom{N}{r}-\binom{m}{r} +t_m$ minor of $M_{J^*}$; call this polynomial $H_{J^*}$. By our choice of basis for $W$, we can require that the columns included in that minor are $w_1, \dots, w_{t_m}$, together with all 
of the $f^*_K$-columns (note that $\pi_{J^*}^{F^*}(w_1,),\dots,\pi_{J^*}^{F^*}(w_{t_m})$ are linearly independent and the
vectors $f_K$ lie in $\ker \pi_{J^*}^{F^*}(w_i)$).
Since $H_{J^*}$ is non-zero for the specific basis $F^*$, it must in fact be a non-zero polynomial in the variables $f_{ij}$. Furthermore, whenever $H_{J^*}(F) \not= 0$, it is true that $\dim \pi^F_{J^*}(W) = t_m$.

We have found a suitable polynomial witness $H_{J^*}$ for a particular $J^* \in \sbsts{N}{m}$. Let $J \in \binom{[N]}{m}$ be arbitrary, and fix a permutation $\sigma: [N] \rightarrow [N]$ with $\sigma(J^*)=J$. If we take $\sigma$ to act on the columns of $F$, it induces a permutation of the variables $f_{ij}$ (we set
$\sigma(f_{ij}) = f_{i \sigma(j)}$)
and thus an automorphism of the polynomial ring generated by the $f_{ij}$'s. 

Consider the matrix $\sigma(M_{J^*})$, by which we mean the matrix resulting when  this polynomial automorphism is applied to the entries of $M_{J^*}$. The $w_i$-columns are unchanged. For $K = \{k_1, \dots, k_r\} \in \sbsts{N}{r}$ we have 
\begin {align*}
\sigma(f_K) & = \sigma ( f_{k_1} \wedge \dots \wedge f_{k_r})\\
& = \sigma(f_{k_1}) \wedge \dots \wedge \sigma (f_{k_r}) \\
& = f_{\sigma(k_1)} \wedge \dots \wedge f_{\sigma(k_r)} = \pm f_{\sigma(K)}.
\end{align*} 
By our choice of the permutation $\sigma$, we have  $ K \not\subseteq J^*$ exactly when $\sigma(K) \not\subseteq J$, so the columns of $\sigma(M_{J^*})$ are a basis for $\ker \pi^F_J$. Finally, set $H_J = \sigma(H_{J^*})$. Then $H_J$ is a non-zero polynomial. It is also a $\binom{N}{r}-\binom{m}{r} +t_m$ by $\binom{N}{r}-\binom{m}{r} +t_m$ minor of the matrix $\sigma(M_{J^*})$. When $H_J(F) \not= 0$, then $\dim \pi^F_J(W) \geq t_m$. Since $t_m$ was chosen to be the maximum possible dimension of a projection of $W$ onto an $m$-dimensional subspace of $V$, in fact $H_J \not= 0$ implies $\dim \pi^F_J(W) = t_m$.

Finally, take $H$ to be the product of $\det F$ and all the $H_J$'s found by the process described above, as $m =|J|$ varies from 1 to $N-1$. 
\end{proof}

%
%

\subsection{Dimensional fractions}\label{sec:dimfracs}

Let $V=\R^n$ and let $W$ be a subspace of $\expow^rV$.   We will prove bounds on the size of subspaces obtained from projecting $W$ onto a subspace of $V$, or wedging with an exterior power of $V$.  Our measure of size will be the \textit{dimensional fraction} 
\[
\frac{\dim W}{\dim \expow^r V} =\frac{\dim W}{\binom{n}{r}}
\]
occupied by a subspace $W \subseteq \expow^r V$. 

Let us begin with projections: our first goal will be to show that there exist projections that preserve the dimensional fraction.  It will be helpful to  consider projections alongside an analogous operation on hypergraphs:
for an $a$-uniform hypergraph  $\hA \subseteq \sbsts{n}{a}$ and $B \in \sbsts{n}{a+b}$, define the \textit{restriction} $\rho_B(\hA) = \{A \in \hA\,:\,A \subseteq B \}$ (i.e.~the subgraph induced by $B$).   The connection between projections and restrictions is given by
$$\pi_B(F(\mathcal A))=F(\rho_B(\mathcal A));$$
in other words, projecting a monomial space on to the subspace generated by $\{f_i:i\in B\}$ corresponds to taking the restriction of the corresponding hypergraph to $B$. 

\textcolor{black}{We define the $\textit{density}$ of an $r$-uniform hypergraph $\hA \subseteq \sbsts{n}{r}$ to be $|A|/\binom{n}{r}$.} The following simple lemma shows that uniform hypergraphs have projections that preserve density.

\begin{lemma}\label{lem:hyperproj} Fix non-negative $a,b,n$ with $a\leq b \leq n$. Let $\hA \subseteq \sbsts{n}{a}$ be an $a$-uniform hypergraph. Then
\[
\max_{B \in \sbsts{n}{b}} \frac{|\pi_B(\hA)|}{\binom{b}{a}} \geq \frac{|\hA|}{\binom{n}{a}}. 
\]
\end{lemma}
\begin{proof}
Count pairs $(A,B)$ with $A \in \hA$, $B \in \sbsts{n}{b}$, and $A \subseteq B$:
\begin{equation*}
|\hA| \binom{n-a}{b-a} = \sum_{B \in \sbsts{n}{b}} | \pi_B(\hA) | \leq \binom{n}{b} \max_{B \in \sbsts{n}{b}} |\pi_B(\hA)|.
\end{equation*}
The first expression follows from choosing $A$ first; the second, from choosing $B$ first. Then divide by $\binom{n}{b}\binom{b}{a} = \binom{n}{a}\binom{n-a}{b-a}$.  Alternatively, simply note that, choosing a $b$-set $B$ uniformly at random, the expected number of edges in the restriction $\pi_B(\mathcal A)$ is
$(\binom ba/\binom na)|\mathcal A|$.
\end{proof}

Let us show that the bound of Lemma \ref{lem:hyperproj} implies a corresponding bound for dimensional fractions of projections. 
Fix $F \in \GL_n(V)$. For $J \subseteq [n]$, define the projection $\pi^F_J: V \rightarrow V_J$ by~\eqref{eq:pidef}. 
Abusing notation, we also write $\pi^F_J: \expow^r V \rightarrow \expow^r V_J$ for the linear map  
defined by $\pi^F_J(f_A) = \expow_{a \in A} \pi^F_J(f_a)$. 
Note that 
\begin{equation}\label{eq:monoproj}
\pi^F_J(f_A) = \begin{cases}
f_A & A \subseteq J, \\
0 & \textrm{otherwise.}
\end{cases}
\end{equation}

\begin{lemma}\label{prop:extproj}
Suppose that $0<r\le n-d$. Let $W \subseteq \expow^r V$ be a linear subspace and $F \in \GL_n (\R)$.  Then
\[
\max_{J \in \sbsts{n}{n-d}}\frac{\dim \pi^F_{J}(W)}{\binom{n-d}{r}} \geq \frac{\dim W}{\binom{n}{r}}.
\]
\end{lemma}

\begin{proof}  Let $J \in \sbsts{n}{n-d}$. By equations~\eqref{eq:colex} and~\eqref{eq:monoproj}, the restriction $\rho_J(\hH(W))$ of the initial hypergraph of $W$ is contained in the initial hypergraph of the projection $\pi^F_J(W)$. That is, 
$ \rho_J(\hH(W)) \subseteq \hH(\pi^F_J(W))$. 
By Lemmas~\ref{lem:inithyp}  and \ref{lem:hyperproj},
\begin{align*}
\max_{J \in \sbsts{n}{n-d}}\frac{\dim \pi^F_{J}(W)}{\binom{n-d}{r}} 
& = \max_{J \in \sbsts{n}{n-d}}\frac{|\hH(\pi^F_{J}(W))|}{\binom{n-d}{r}}\\
& \geq \max_{J \in \sbsts{n}{n-d}}\frac{|\rho_{J}(\hH(W))|}{\binom{n-d}{r}}\geq \frac{\hH (W)}{\binom{n}{r}} =\frac{\dim W}{\binom{n}{r}}.
\end{align*}
\end{proof}

The existence of generic subspaces implies that a typical projection achieves the bound of Lemma \ref{prop:extproj}:

\begin{corollary}\label{cor:genericproj} Fix $ 0 < d \leq n$. Let $W \subseteq \expow^r V$ be a linear subspace and $F \in \GL_n (\R)$. Then there exists a nonempty Zariski open set of $F \in \GL_n(\R)$ satisfying  
\[
\frac{\dim \pi^F_{[n-d]}(W)}{\binom{n-d}{r}} \geq \frac{\dim W}{\binom{n}{r}}.
\]
\end{corollary}
\begin{proof} By Lemma~\ref{lem:constantproj}, for all $F$ outside of the zero set of a particular polynomial, the dimension of $\pi^F_J(W)$ depends only on $|J|$, and thus $\dim \pi^F_{[n-d]}(W) = \max_{J \in \sbsts{n}{n-d}}\dim \pi^F_{J}(W)$.
\end{proof}

We now turn to the behavious of $W$ under wedging with exterior powers of $V$.  We know from Theorem \ref{eq:extshadow} that wedging with exterior powers of $V$ preserves the dimensional fraction.  However, for our application we will need a stronger bound.  We will show (Corollary \ref{cor:genericboth}) that if $W$ has a {\em projection} with large dimensional fraction then wedging with a suitable exterior power of $V$ gives a subspace achieving at least the same dimensional fraction.

We first bound  the dimensional fraction of $W \wedge \expow^d V$ in terms of the average dimensional fraction of a projection onto an $(n-d)$-dimensional subspace. 

\begin{lemma}\label{prop:increaseboth} 
Suppose that $0<  r \leq n-d \leq n$.   
Let $W \subseteq \expow^r V$ and $F \in \GL_n (\R)$.   Then
\[
\frac{\dim \left(W \wedge \expow^d V \right) }{\binom{n}{r+d}} \geq \frac{1}{\binom{n}{n-d}}\sum_{J \in \sbsts{n}{n-d}}\frac{\dim \pi^F_J(W)}{\binom{n-d}{r}}.
\]
\end{lemma}

\begin{proof}
Recall that the columns $\{f_1,\dots, f_n\}$ of $F$ form a basis for $V$ and that we write $f_K = f_{k_1}\wedge\dots\wedge f_{k_d}$ when $K=\{k_1,\dots,k_d\}$ with the elements listed in increasing order. We know that $W\wedge \expow^dV= \myspan\left\{W \wedge f_K \,:\, K \in \sbsts{n}{d}\right\}$, and so
$$\mathcal H\left(W\wedge\expow^dV\right)\supseteq \bigcup_{K\in \binom{[n]}{d}} \mathcal H(W\wedge f_K).$$
For $K\in \binom{[n]}{d}$, we have $v\wedge f_K=\pi_{[n]\setminus K}^F(v)\wedge f_K$ for all $v\in V$, and so
$\dim(W\wedge f_K)=\dim \pi_{[n]\setminus K}^F(W)$.  Furthermore, 
$$\mathcal H(W\wedge f_K)=\mathcal H\left( \pi_{[n]\setminus K}^F(W) \wedge f_K\right)=\left\{J\cup K: J\in  \mathcal H\left(\pi_{[n]\setminus K}^F(W)\right)\right\}.$$
Each set $S\in \mathcal H(W\wedge\expow^dW)$ has size $r+d$, and can occur in at most $\binom{r+d}{r}$ distinct families $\mathcal H(W\wedge f_K)$ (as there are only $\binom{r+d}{r}$ sets $K\subseteq S$ of size $d$).  Thus
\begin{align*}
\left|\mathcal H\left(W\wedge\expow^dV\right)\right|
&\ge\frac{1}{\binom{r+d}{r}}\sum_{K\in \binom{[n]}{d}}\left|\mathcal H\left( \pi_{[n]\setminus K}^F(W) \wedge f_K\right)\right|\\
&=\frac{1}{\binom{r+d}{r}}\sum_{K\in \binom{[n]}{d}}\dim \left(\pi_{[n]\setminus K}^F(W)\right).
\end{align*}
By Lemma~\ref{lem:inithyp}, $\dim\left(W\wedge \expow^dV\right)=\left|\mathcal H(W\wedge\expow^dV)\right|$. Since $\binom{n}{n+d}\binom{n+d}{r} = \binom{n}{n-d}\binom{n-d}{r}$, the result now follows.
\end{proof}

Once again, we use the existence of generic subspaces to obtain the desired bound.

\begin{corollary}\label{cor:genericboth} Let $V=\R^n$ and fix $ 0 < r < r+ d \leq n $. Let $W \subseteq \expow^r V$ be a linear subspace. Then there exists a nonempty Zariski open set of $F \in \GL_n(\R)$ satisfying  
\[
\frac{\dim \left( W \wedge \expow^d V\right)}{\binom{n}{r+d}}
\geq \frac{\dim \pi^F_{[n-d]}(W)}{\binom{n-d}{r}}
=\max_{J^*,F^*}\frac{ \dim \pi^{F^*}_{J^*} (W)}{\binom{n-d}{r}},
\]
where the maximum is taken over all $J^* \in \sbsts{N}{n-d}$ and $F^* \in \GL_N(\R)$. 
\end{corollary}
\begin{proof} By Lemma~\ref{lem:constantproj}, for all $F$ outside of the zero set of a particular polynomial, the dimension of $\pi^F_J(W)$ depends only on $|J|$, and thus $\dim \pi^F_{[n-d]}(W) = \frac{1}{\binom{n}{n-d}}\sum_{J \in \sbsts{n}{n-d}}\dim \pi^F_J(W)$.
The inequality then follows from Lemma \ref{prop:increaseboth}.
\end{proof}

\section{Two Families Theorems} 
\label{sec:twofam}

%
%

\subsection{Context and consequences}

Bollob\'as's Two Families Theorem \cite{Bollobas65} has been rediscovered in different forms and proved in several different ways (see~\cites{JP71,Katona74,Tarjan75,Lovasz77,AK85,Alon85,GST84,K84,Kalai84,Blokhuis89}, Tuza's surveys \cites{TuzaSurveyI,TuzaSurveyII} of applications, and the expository discussions in  Bollob\'as~\cite{WhiteBook}*{Chapters 9 and 15}, F\"{u}redi~\cite{Furedi88}*{Sections 1 and 2}, Anderson~\cite{Anderson2002}*{Section 1.3}, Babai and Frankl~\cite{BF92}*{Sections 5.1 and 6.2}, Kalai~\cite{KalaiBlog}, Matou\v{s}ek~\cite{Matousek}*{Miniature 33}, Jukna~\cite{Jukna2011}*{Section 9.2.2},  Frankl and Tokushige~\cite{FT}*{Sections 26.2--4}, and Gerbner and Patk\'{o}s~\cite{GP}*{Section 1.1}). 
The simplest version of the Two Families Theorem is perhaps the following: 

\begin{theorem}[Uniform Two Families] \label{thm:ciuniform} 
Let $(A_1,B_1),\dots,(A_m,B_m)$ 
be a sequence of pairs of sets with $|A_i| = a$ and $|B_i| = b$ for every  $i$.
Suppose that
\begin{enumerate}[label=\textup{(\roman*)}]
\item $A_i \cap B_i = \emptyset$ for $1 \leq i \leq m$, and
\item \label{it:intersect} $A_i \cap B_j \not= \emptyset$ for $i \not= j$.
\end{enumerate}
Then $m \leq \binom{a+b}{a}$.  Furthermore, if $m=\binom{a+b}{a}$ then there is some set $S$ of cardinality $a+b$ such that
the $A_i$ are all subsets of $S$ of size $a$, and $B_i=S\setminus A_i$ for each $i$.
\end{theorem}
A striking feature of this theorem is that the upper bound depends only on $a$ and $b$, and not on the size of the ground set (compare Theorem \ref{thm:EKR}).

There are two standard approaches to proving the Two Families Theorem, each of which exemplifies important methods in the field and leads to a different generalization.  
One approach is combinatorial (see Bollob\'as \cite{Bollobas65}, or the elegant counting argument due to Katona~\cite{Katona74}).  With this approach, the assumption that the sets in each pair have the same sizes can be relaxed. When $|A|=a$ and $|B|=b$, we will say that the pair $(A,B)$ has \textit{profile} $(a,b)$ and \textit{profile sum} $a+b$. Note that when $|X|=a+b$, there are $\binom{a+b}{b}$ complementary pairs $(A,B)$ with profile $(a,b)$ and $A,B \subseteq X$. Bollob\'as's original result~\cite{Bollobas65} is equivalent to~Theorem~\ref{thm:ciweighted}, which weights each pair of sets by the nominal fraction of the set of pairs with  matching union and profile that it occupies.  

\begin{theorem}[Weighted Two Families] \label{thm:ciweighted} 
Let $(A_1,B_1),\dots,(A_m,B_m)$ 
be a finite collection of pairs of finite sets. Let $a_i=|A_i|$ and $b_i=|B_i|$ for  $1\leq i\leq m$.
Suppose that
\begin{enumerate}[label=\textup{(\roman*)}]
\item $A_i \cap B_i = \emptyset$ for $1 \leq i \leq m$, and
\item $A_i \cap B_j \not= \emptyset$ for $i \not= j$. 
\end{enumerate}
Then 
\begin{equation}\label{abbound}
\sum_{i=1}^m \frac{1}{\binom{a_i+b_i}{a_i}} \leq 1. 
\end{equation}
Furthermore, if equality is achieved, then there is some finite set $S$ and $0 \leq a_0 \leq |S|$ such that
the $A_i$ are the subsets of $S$ of size $a_0$ and $B_i=S\setminus A_i$ for each $i$.
\end{theorem}

A second approach, introduced by 
Lov\'asz~\cite{Lovasz77}, uses exterior algebra methods.  This method gives an elegant argument that naturally extends to subspaces of a finite dimensional vector space; a set system version of Two Families follows immediately (using the standard construction illustrated in Corollary~\ref{cor:combo}).  Frankl~\cite{F82} used a similar approach and noted that this method also allows the relaxation of condition~\ref{it:intersect}: instead of requiring $A_i$ and $B_j$ to intersect for all pairs with $i\not=j$, we insist only that the intersection is non-trivial when $i<j$. Proofs of this form of the Two Families Theorem also appeared in \cites{K84, Alon85, AK85}. 

\begin{theorem}[Uniform Skew Subspace Two Families]\label{thm:skewspaces} Let $(A_1,B_1),\dots,(A_m,B_m)$ 
be pairs of non-trivial subspaces of $V=\R^N$. Suppose that $\dim A_i \le a$ and $\dim B_i \le b$ for $1\leq i\leq m$, and 
\begin{enumerate}[label=\textup{(\roman*)}]
\item \label{it:vssdisjoint} $\dim(A_i \cap B_i) = 0 $ for $1 \leq i \leq m$, and
\item \label{it:vssintersect} $\dim(A_i \cap B_j) > 0$ for $1 \leq i<j \leq m$.
\end{enumerate}
Then $m \leq \binom{a+b}{a}$.
\end{theorem}

A version for hypergraphs follows immediately.\footnote{Note that there is not a unique extremal hypergraph for Corollary~\ref{cor:wci}: for example, $B_1$ can be any $b$-element set disjoint from $A_1$.}

\begin{corollary}[Uniform Skew Two Families] \label{cor:wci} 
Let $(A_1,B_1),\dots,(A_m,B_m)$ 
be a sequence of pairs of sets with $|A_i|=a$ and $|B_i|=b$ for every  $i$.
Suppose that
\begin{enumerate}[label=\textup{(\roman*)}]
\item \label{it:skewdisjoint} $A_i \cap B_i = \emptyset$ for $1 \leq i \leq m$, and
\item \label{it:skewintersect} $A_i \cap B_j \not= \emptyset$ for $1 \leq i < j \leq m$.
\end{enumerate}
Then $m \leq \binom{a+b}{b}$.
\end{corollary}

Thus there are two completely different extensions of the Two Families Theorem: in one case, the set pairs are weighted according to their size; and in the other, the intersection condition is weakened to a skew intersection condition. 
It is natural to wonder if the Two Families Theorem can be extended in both these directions at once.  In other words, is there a Two Families Theorem that has both weights and a skew hypothesis?
For example, Tuza~\cite{TuzaSurveyI}*{Question 12} asked whether linear algebra techniques can be used to prove Two Families theorems in cases where the two families are not of constant profile. 

The main result of this section is the following, which shows that under suitable conditions it is indeed possible to combine the two directions of generalization.  We first state the result for subspaces.

\begin{theorem}
\label{thm:subspaces} 
Let $(A_1,B_1),\dots,(A_m,B_m)$ 
be pairs of non-trivial subspaces of a finite-dimensional real vector space. Write $a_i= \dim A_i$ and $b_i= \dim B_i$ for $1\leq i\leq m$.  
Suppose that
\begin{enumerate}[label=\textup{(\roman*)}]
\item \label{it:vsdisjoint} $\dim(A_i \cap B_i) = 0 $ for $1 \leq i \leq m$, 
\item \label{it:vsintersect} $\dim(A_i \cap B_j) > 0$ for $1 \leq i < j \leq m$, and
\item \label{it:vsorder} $a_1 \leq a_2 \leq \dots\leq a_m$ and $b_1 \geq b_2 \geq \dots \geq b_m$.
\end{enumerate}
Then 
\begin{equation}\label{eq:all}
\sum_{i=1}^m \frac{1}{\binom{a_i+b_i}{a_i}} \leq 1.
\end{equation}
\end{theorem}
We prove this in the next subsection.  The proof 
works in varying levels of the exterior algebra and over vector spaces of varying dimension. For this, we will use the upwards Local LYM inequality of Section~\ref{sec:LYMup} and the projection and wedging bounds of Section~\ref{sec:projections}.

A combinatorial version of Theorem \ref{thm:subspaces} follows immediately via a standard construction: 
\begin{corollary}[Weighted Skew Two Families]\label{cor:combo} Let $(A_1,B_1),\dots,(A_m,B_m)$ 
be pairs of finite non-empty sets. Write $a_i=|A_i|$ and $b_i=|B_i|$ for $1\leq i\leq m$.  
Suppose that
\begin{enumerate}[label=\textup{(\roman*)}]
\item \label{it:disjoint} $A_i \cap B_i = \emptyset$ for $1 \leq i \leq m$, 
\item $A_i \cap B_j \not= \emptyset$ for $1 \leq i < j \leq m$, and
\item \label{it:order} $a_1 \leq a_2 \leq \dots\leq a_m$ and $b_1 \geq b_2 \geq \dots \geq b_m$. 
\end{enumerate}
Then 
\begin{equation}
\sum_{i=1}^m \frac{1}{\binom{a_i+b_i}{a_i}} \leq 1.
\end{equation}
\end{corollary}
\begin{proof}[Proof] Let $N \in \N$ be large enough that we may assume $A_i,B_i \subseteq [N]$ for $1 \leq i \leq m$. Let $\{e_1,\dots, e_N\}$ be the standard basis of $\R^N$. Map each set $A_i$ to the subspace $A'_i = \myspan\{e_a\,:\, a\in A_i\}\subseteq \R^N$ and each $B_i$ to the subspace $B'_i = \myspan\{e_b\,:\,b \in B_i\}\subseteq \R_N$. Then $\dim A'_i=a_i$, $\dim B'_i = b_i$, and the hypotheses of Theorem~\ref{thm:subspaces} are satisfied by these subspaces. 
\end{proof}

A bound of form \eqref{eq:all} does not hold 
for arbitrary families of pairs satisfying a skew intersection condition without adding some restriction on the set sizes, as the following examples show.

\begin{example}[Babai and Frankl~\cite{BF92}*{Exercise 5.1.1}] \label{ex:death}
List all pairs $(A,A^C)$ with $A \in 2^{[n]}$, sorted by decreasing cardinality of the first element. This ``death'' example, in which the $a_i$'s decrease as the $b_i$'s increase,  satisfies~\ref{it:skewdisjoint} and ~\ref{it:skewintersect}, but
\[
\sum_{1=1}^{2^n} \frac{1}{\binom{n_i}{a_i}} = \sum_{j=0}^{n}\frac{\binom{n}{j}}{\binom{n}{j}}= n+1.
\]
\end{example}

\begin{example}\label{ex:log} Keeping one family of sets of constant size is also insufficient. Set
 $(A_i,B_i) = (\{i\},[i-1])$ for $1 \leq i \leq n$. Now
\[
\sum_{i=1}^n \frac{1}{\binom{n_i}{a_i}} = \sum_{i=1}^n \frac{1}{i}\sim \log n.
\]
\end{example}

Returning to subspaces, it
is natural to wonder whether a weighted Two Families Theorem holds under the full symmetric cross-intersecting hypothesis. Theorem~\ref{thm:subspaces} allows some progress: 

\begin{corollary}
\label{cor:spaceconstsum} Let $n\ge2$, and suppose that $(A_1,B_1),\dots,(A_m,B_m)$ 
are pairs of non-trivial subspaces of $V=\R^N$ such that $a_i+b_i=n$  for $1 \leq i \leq m$, where $a_i=\dim A_i$ and $b_i= \dim B_i$.    
Suppose that
\begin{enumerate}[label=\textup{(\roman*)}]
\item \label{it:vssdisjoint2} $\dim(A_i \cap B_i) = 0 $ for $1 \leq i \leq m$, and
\item \label{it:vssintersect2} $\dim(A_i \cap B_j) > 0$ for $1 \leq i, j \leq m$ with $i\neq j$.
\end{enumerate}
Then 
\begin{equation*}
\sum_{i=1}^m \frac{1}{\binom{n}{a_i}} \leq 1.
\end{equation*}
\end{corollary}
\begin{proof}
Permute the subscripts of the pairs $(A_i,B_i)$ so that the $A_i$'s are listed in increasing order of dimension; since our cross-intersecting hypothesis~\ref{it:vssintersect} is symmetric, we can do so. Because the profile sums $a_i+b_i=n$ are constant, 
the resulting system satisfies all hypotheses of Theorem~\ref{thm:subspaces}. 
\end{proof}

The following also follows straighforwardly from Theorem~\ref{thm:subspaces}.

\begin{corollary}
\label{cor:weightedspaces} Let $n\ge2$ and suppose that 
$(A_1,B_1),\dots,(A_m,B_m)$ 
are pairs of non-trivial subspaces of $V=\R^N$. Write $a_i=\dim A_i$ for $1\leq i\leq m$, and let $b = \max _i \dim(B_i)$.   Suppose that
\begin{enumerate}[label=\textup{(\roman*)}]
\item \label{it:vssdisjoint3} $\dim(A_i \cap B_i) = 0 $ for $1 \leq i \leq m$, and
\item \label{it:vssintersect3} $\dim(A_i \cap B_j) > 0$ for $1 \leq i, j \leq m$ with $i\neq j$.
\end{enumerate}
Then 
\begin{equation*}
\sum_{i=1}^m \frac{1}{\binom{a_i+b}{a_i}} \leq 1.
\end{equation*}
\end{corollary}
\begin{proof}
First, permute the subscripts of the pairs $(A_i,B_i)$ of spaces so that the $A_i$'s are listed in increasing order of dimension; since our cross-intersecting hypothesis~\ref{it:vssintersect} is symmetric, we can do so. 

Let $a=\max_i a_i$, and embed the entire system in ${\mathbb R}^{a+b}$. For each $b_i < b$, extend $B_i$ by including in it $b-b_i$ linearly independent vectors outside $A_i$. 
The resulting system, in which $a_1 \leq \dots \leq a_m$ and $b_i = b$ for $1 \leq i \leq m$,  satisfies the hypotheses of Theorem~\ref{thm:subspaces}. 
\end{proof}
Note that the proof of Corollary~\ref{cor:weightedspaces} does not use the full symmetric cross-intersecting condition: the argument goes through as long as the pairs of spaces are fully cross-intersecting between distinct profiles, but possibly only weakly cross-intersecting (with respect to some ordering) within the collections of pairs with the same profile. 

\subsection{Proof of Theorem~\ref{thm:subspaces}}\label{sec:proofs}

First, a definition: for a subspace $C \subseteq V=\R^N$ with basis $\{c_1, \dots, c_d\},$ we define the  {\em $d$-blade} 
\begin{equation}\label{eq:defblade}
v_C = c_1 \wedge \dots \wedge c_d \in \expow^d V.
\end{equation}
Although $v_C$ is only determined up to a non-zero constant, $\myspan\{v_C\}$ is a well-defined one-dimensional subspace of $\expow^d V$.

We now sketch our strategy. The hypotheses of Theorem~\ref{thm:subspaces} allow both the $a_i$'s and the profile sums $n_i=a_i+b_i$ to vary in~$i$.  Because the profile sums can vary, we will want to vary the dimension of the underlying vector space. Because the $a_i$'s can vary, we will want to vary the exterior degree as well. We will deal with this by inductively constructing a sequence of subspaces $Z_i$, where $Z_i$ lies in  $\expow^{a_i} \R^{n_i}$. The space $Z_i$ encodes the intersection structure of the pairs $(A_1,B_1), \dots, (A_i,B_i)$ and will satisfy
\begin{equation}\label{eq:Zsketch}
\frac{\dim Z_i}{\binom{n_i}{a_i}} \geq \sum_{j=1}^i \frac{1}{\binom{n_j}{a_j}}.
\end{equation}
  
%
%

\begin{proof}[Proof of Theorem~\ref{thm:subspaces}] 
The main step in the proof lies in associating to the space $Z_i \subseteq \expow^{a_i} \R^{n_i}$ a suitable space $Y_i \subseteq \expow^{a_{i+1}} \R^{n_{i+1}}$ such that
\begin{equation*} 
\frac{\dim Y_i}{\binom{n_{i+1}}{a_{i+1}}} \geq \frac{\dim Z_i}{\binom{n_i}{a_i}},
\end{equation*}
and $Y_i$ does not contain the $a_{i+1}$-blade corresponding to the space $A_{i+1}$.
We then extend $Y_i$ by the $a_{i+1}$-blade, increasing its dimension by 1, to obtain $Z_{i+1}$ satisfying~inequality \eqref{eq:Zsketch} for $i+1$. Continuing through to $i=m$ and noting that $\dim Z_m \leq \binom{n_m}{a_m}$ gives the desired inequality.  

Rather than defining spaces $Z_i$, $Y_i$ directly, we define them as projections of a sequence of spaces $W_i$ sitting in appropriate exterior powers of the ground space $V$.
We recursively construct the sequence $W_i \subseteq \expow^{a_i} V$ by setting $W_0 = \{0\}$ and, for $0 \leq i \leq m-1$,
\begin{equation}\label{eq:Wdef}
W_{i+1} = \myspan\left\{ W_i \wedge \expow^{a_{i+1}-a_i} V, \, v_{A_{i+1}}\right\}.
\end{equation}
We will fix a suitable basis $F$ for $V$ and use it to define a sequence of subspaces $V_{n_i}=\pi^F_{[n_i]}(V)$ of $V$.  Since $V_{n_i}$ is generated by the first $n_i$ basis elements of $F$, we have $\dim V_{n_i} = n_i$.  We then define
$Z_i$ as the projection of $W_i$ on $\expow^{a_i}V_{n_i}$, and take $Y_i$ to be the projection of $W_i$ onto $\expow^{a_i}V_{n_{i+1}}$, wedged with
$\expow^{a_{i+1}-a_i}V_{n_{i+1}}$. That is, $Y_i$ is a subspace of $\expow^{a_i+1} V_{n_{i+1}}$, as is $Z_{i+1}$. As we prove our chain of inequalities, we will need to relate the dimensions of $Z_i$ and $Z_{i+1}$; the space $Y_i$ provides an intermediate step.

Let us give precise definitions  of the spaces described above.
Let $C_i = \myspan\{A_i,B_i\}$, and let $n_i = \dim C_i = a_i+b_i$ (by hypothesis~\ref{it:vsdisjoint}). By Lemmas~\ref{lem:nocollapse} and~\ref{lem:constantproj}, there is a Zariski open set of bases $\{f_1,\dots,f_N\}$ for $V$ that satisfy the following:  for every  $J \subseteq [N]$ and all $1 \leq i<j  \leq m$,
\begin{subequations}
\begin{align}
\label{eq:Cdim}  \dim(\pi^F_J(C_i)) & = \min \{ n_i, |J|\},\\
\label{eq:intersect} \dim(\pi^F_J(A_i \cap  B_j)) & = \min \{ \dim (A_i \cap B_j), |J|\},\\
\label{eq:constproj} \dim \pi^F_{J}(W_i) & =  t_{i,|J|},
\end{align}
\end{subequations}
where $t_{i,|J|}$ is the maximum dimension of $\pi_{J^*}^F(W_i)$ over all choices of $F$ and $J^*$ with $|J^*|=|J|$. 
Fix one such generic basis $F$, and note that it will satisfy  Corollaries~\ref{cor:genericproj} and ~\ref{cor:genericboth}.  Let
\begin{equation}\label{eq:YZdef}
Z_i= \pi^F_{[n_i]}(W_i), \  X_i= \pi^F_{[n_{i+1}]}(W_i) \
\textrm{ and }\
Y_{i} = X_i \wedge \expow^{a_{i+1}-a_i} V_{[n_{i+1}]}.
\end{equation} 
Thus $Z_i$ is a subspace of $\expow^{a_i}V_{n_i}$, while $X_i $ is a subspace of $\expow^{a_{i}}V_{n_{i+1}}$ and $Y_i$ is a subspace of $\expow^{a_{i+1}}V_{n_{i+1}}$. 

We will  verify that for $0 \leq i \leq m-1$ 
\begin{equation}\label{eq:Zdim}
\dim Z_{i+1}  = \dim Y_i +1\\
\end{equation}
and
\begin{align}
\label{eq:phidim} 
\frac{\dim Y_{i}}{\binom{n_{i+1}}{a_{i+1}}} & \geq \frac{\dim Z_{i}}{\binom{n_{i}}{a_{i}}}.
\end{align}
We then complete the proof by applying \eqref{eq:Zdim} and \eqref{eq:phidim} in alternation until the final result is reached:
\begin{align*}
1 & \geq \frac{\dim Z_m}{\binom{n_m}{a_m}} = \frac{1+\dim Y_{m-1}}{\binom{n_m}{a_m}}\geq \frac{1}{\binom{n_m}{a_m}} + \frac{\dim Z_{m-1}}{\binom{n_{m-1}}{a_{m-1}}} =
\cdots \geq \sum_{i=1}^{m} \frac{1}{\binom{n_i}{a_i}}.
\end{align*}

\textit{Proof of~\eqref{eq:Zdim}}: 
By the definitions \eqref{eq:Wdef} and \eqref{eq:YZdef} of $W_{i+1}$ and $Z_{i+1}$,
\begin{align*}
Z_{i+1} &=  \pi^F_{[n_{i+1}]}(W_{i+1})\\
&= \myspan\left(\pi^F_{[n_{i+1}]}(v_{A_{i+1}}),  \pi^F_{[n_{i+1}]}\left(W_i \wedge \expow^{a_{i+1}-a_i} V_{[n_{i+1}]}\right)\right)\\
&= \myspan\left(\pi^F_{[n_{i+1}]}(v_{A_{i+1}}),  \pi^F_{[n_{i+1}]}(W_i) \wedge \expow^{a_{i+1}-a_i} V_{[n_{i+1}]}\right)\\
&= \myspan(\pi^F_{[n_{i+1}]}(v_{A_{i+1}}), Y_i).
\end{align*}
So it will suffice to check that $\pi^F_{[n_{i+1}]}\left(v_{A_{i+1}} \right) \not\in Y_i$. 
By hypothesis (i), we have $v_{A_{i+1}}\wedge v_{B_{i+1}}\ne 0$.  Since
$n_{i+1}=a_{i+1}+b_{i+1}$, it follows from \eqref{eq:Cdim} that
\begin{equation}\label{a0}
\pi^F_{[n_{i+1}]}(v_{A_{i+1}}) \wedge \pi^F_{[n_{i+1}]}(v_{B_{i+1}}) \not=0.
\end{equation}

Now consider $y\in W_i\wedge \expow^{a_{i+1}-a_i}V$.  For $h<i+1$, hypothesis (ii) implies that $v_{A_h}\wedge v_{B_{i+1}}=0$.  Since
(by \eqref{eq:Wdef}), $y$ is a linear combination of elements $\{v_{A_h}\wedge \expow^{a_{i+1}-a_h}V:h\le i\}$, it follows that $y\wedge v_{B_{i+1}}=0$.  Thus
\begin{equation}\label{a1}
 \pi^F_{[n_{i+1}]}(y) \wedge \pi^F_{[n_{i+1}]}(v_{B_{i+1}}) =0.
\end{equation}
Equation \eqref{eq:Zdim} now follows from \eqref{a0} and \eqref{a1}.

\medskip

\textit{Proof of~\eqref{eq:phidim}}: Our argument depends on how $(a_{i+1},b_{i+1})$ is related to $(a_{i},b_{i})$.\\

\begin{itemize}
\item \textbf{Profile unchanged.} When $(a_{i+1},b_{i+1})=(a_{i},b_{i})$, we also know $n_{i+1}=n_{i}$ and $Y_{i}=Z_i$, so~(\ref{eq:phidim}) follows immediately.\\

\item \textbf{Profile sum constant.} When $(a_{i+1},b_{i+1})=(a_i+c,b_i-c)$ for some $c>0$, we have $n_{i+1}=n_{i}$ and $Y_{i}= Z_{i} \wedge \left(\expow^c V_{[n_{i}]} \right)$, so Lemma~\ref{prop:extLYM}  gives~(\ref{eq:phidim}).\\

\item \textbf{$B_i$'s shrink faster.} When $(a_{i+1},b_{i+1}) = (a_{i}+c,b_{i}-c-d)$ for some $c \geq 0$ and $d>0$, we have $n_{i+1}=n_{i}-d$.
By Lemma~\ref{prop:extLYM}, 
\begin{align*}
\frac{\dim Y_{i}}{\binom{n_{i+1}}{a_{i+1}}}
= \frac{\dim \left(\pi^F_{[n_i-d]}(W_i)\wedge \expow^c V_{[n_i-d]} \right)}{\binom{n_{i}-d}{a_{i}+c}} 
\geq \frac{\dim \pi^F_{[n_i-d]}(W_i)}{\binom{n_{i}-d}{a_{i}}}.
\end{align*}
Since $\pi_{[n_i-d]}^F(W_i)=\pi_{[n_i-d]}^F\left(\pi_{[n_i]}^F(W_i)\right)$,  Corollary \ref{cor:genericproj} and our generic choice of $F$ imply
\begin{align*}
\frac{\dim \pi^F_{[n_i-d]}(W_i)}{\binom{n_{i}-d}{a_{i}}} 
\geq \frac{\dim \pi^F_{[n_i]}(W_i)}{\binom{n_{i}}{a_{i}}}
= \frac{\dim Z_{i}}{\binom{n_{i}}{a_{i}}},
\end{align*}
Thus \eqref{eq:phidim} holds.\\

\item \textbf{$A_i$'s grow faster.} When $(a_{i+1},b_{i+1}) = (a_{i}+c+d,b_{i}-c)$ for some $c \geq 0$ and $d>0$,  we have $n_{i+1}=n_i+d$.
By Lemma~\ref{prop:extLYM}, 
\begin{align*}
\frac{\dim Y_{i}}{\binom{n_{i+1}}{a_{i+1}}}
&=\frac{\dim \left(\pi^F_{[n_i+d]}(W_i) \wedge \expow^{d+c} V_{[n_i+d]} \right)}{\binom{n_{i}+d}{a_{i}+d+c}}\\
 &\geq \frac{\dim \left(\pi^F_{[n_i+d]}(W_i) \wedge \expow^d V_{[n_i+d]} \right)}{\binom{n_{i}+d}{a_{i}+d}}.
\end{align*}
Since $\pi_{[n_i]}^F(W_i)=\pi_{[n_i]}^F\left(\pi_{[n_i+d]}^F(W_i)\right)$ we can apply Corollary \ref{cor:genericboth} and equation \eqref{eq:constproj} 
to obtain
\begin{align*}
\frac{\dim \left(\pi^F_{[n_i+d]}(W_i) \wedge \expow^d V_{[n_i+d]} \right)}{\binom{n_{i}+d}{a_{i}+d}}
\geq \frac{\dim \left(\pi^F_{[n_i]}(W_i)\right)}{\binom{n_{i}}{a_{i}}}
 = \frac{\dim Z_{i}}{\binom{n_{i}}{a_{i}}}.
\end{align*}
\end{itemize}
Thus \eqref{eq:phidim} holds.
\end{proof}


\section{\textcolor{black}{An additional application \label{sec:addapp}}}

In a recent preprint~\cite{GKMNPTX}, Gerbner, Keszegh, Methuku, Abhishek, Nagy, Patk\'{o}s,
Tompkins, and Xiao consider bounding the size of fully cross-intersecting pairs of families of sets, with fixed profile $(a,b)$, under the additional assumption that one of the two families is also $t$-intersecting. For the $t=1$ case, they deduce an upper bound of $\frac{1}{2}\binom{a+b}{a}$ from the weighted skew Two Families Theorem in an earlier version of this paper, and conjecture  that the Erd\"{o}s-Ko-Rado bound of $\binom{a+b-1}{a-1}$ holds~\cite{GKMNPTX}*{Conjecture 2.4}. They also make a more general conjecture, proposing that the number of pairs in such a system is bounded by $AK(a+b,a,t)$~\cite{GKMNPTX}*{Conjecture 2.5}, where  $AK(n,a,t)$ denotes the maximum size of an $a$-uniform $t$-intersecting family $\hA \subseteq [n]$, as determined by Ahlswede and Khachatrian~\cite{AhKh}.

We will prove this conjecture. In fact, our Theorem~\ref{thm:GKsubspaces} below is more general in two ways: it applies to subspace configurations, and the cross-intersecting condition is relaxed to skew. A set system version, Corollary~\ref{cor:GKMNPTXconj}, follows immediately via the standard construction. 

\begin{theorem}\label{thm:GKsubspaces}
Fix positive integers $t \leq a \leq b$. Let $(A_1,B_1),\dots,(A_m,B_m)$ 
be a collection of pairs of subspaces of a real vector space $V$ with $\dim A_i=a$, $\dim B_i=b$ for  $1\leq i\leq m$.
Suppose that
\begin{enumerate}[label=\textup{(\roman*)}]
\item $\dim (A_i \cap B_i) = 0$ for $1 \leq i \leq m$, 
\item $\dim (A_i \cap B_j) >0$ for $1 \leq i < j \leq m$, and
\item $\dim (A_i \cap A_j) \geq t$ for $1 \leq i,j \leq m$. 
\end{enumerate}
Then $m \leq AK(a+b,a,t)$. 
\end{theorem}
\begin{proof} 

First, we note that we may without loss of generality assume $\dim V = n= a+b$. Why? If $\dim V =N > a+b$, we can apply Lemma~\ref{lem:nocollapse} to the space $V$ and the list of subspaces containing  $A_i$, $B_i$, $A_i \cap B_j$, and $A_i \cap A_j$, for all $1 \leq i,j \leq m$.  The result is a basis $F = \{f_1,\dots, f_{N}\}$ for $V$ such that the pairs of subspaces $\left\{\left(\pi^F_{[n]}(A_i),\pi^F_{[n]}(B_i)\right)\,:\, 1 \leq i \leq m\right\}$ of $V_{[n]}=\myspan\{f_1,\dots,f_n\}$ satisfy all the hypotheses of the theorem. In this case we simply replace $V$ by $V_{[n]}$ and replace each $(A_i,B_i)$ by $\left(\pi^F_{[n]}(A_i),\pi^F_{[n]}(B_i)\right)$.

Set $W = \myspan\{v_{A_i}\,:\, 1 \leq i \leq m \}$,
where the the $a$-blade $v_{A_i} \in \expow^a V$ is defined by \eqref{eq:defblade}. Hypotheses (i) and (ii) ensure that the usual exterior algebra argument for Theorem~\ref{thm:skewspaces} goes through, so the $a$-blades $\{v_{A_i}\,:\,1 \leq i \leq m\}$ are linearly independent and $\dim W =m$.

\textcolor{more}{It is also true that $W$ is a $t$-self-annihilating subspace of $\expow^a V$. 
Why? First, hypothesis (iii) and Lemma~\ref{lem:tselfann} ensure that 
\[
(v_{A_i} \innprod y^*) \wedge (v_{A_j} \innprod y^*) = 0
\]
for every decomposable $y^* \in \expow^{<t} V^*$. Given arbitrary $u, w \in W$, 
expand 
\[
u = \sum_{i} \alpha_i v_{A_i}, \qquad v = \sum_{j} \beta_j v_{A_j}.
\]
Then by bilinearity,
\begin{align*}
(u \innprod y^*) \wedge (v \innprod y^*) & = \sum_{i} \sum_{j}
\alpha_i \beta_j\, (v_{A_i} \innprod y^*) \wedge (v_{A_j} \innprod y^*),
\end{align*}
so we have verified the full definition~\eqref{eq:tselfanndef}.}

The desired result now follows from Theorem~\ref{thm:AKspaces}. 
\end{proof}

\begin{corollary}\label{cor:GKMNPTXconj} 
Fix positive integers $t \leq a \leq b$. Let $(A_1,B_1),\dots,(A_m,B_m)$ 
be a collection of pairs of sets with $|A_i|=a$, $|B_i|=b$ for  $1\leq i\leq m$.
Suppose that
\begin{enumerate}[label=\textup{(\roman*)}]
\item $A_i \cap B_i = \emptyset$ for $1 \leq i \leq m$, 
\item $A_i \cap B_j \not= \emptyset$ for $1 \leq i < j \leq m$, and
\item $|A_i \cap A_j| \geq t$ for $1 \leq i,j \leq m$. 
\end{enumerate}
Then $m \leq AK(a+b,a,t)$. 
\end{corollary}
\begin{proof} Assume without loss of generality that $A_i, B_i \subset [N]$ for some $N \in \N$, and let $E = \{e_1,\dots, e_N\}$ be the standard basis of $\R^N$. Let 
\[
U_i = \myspan\{e_k \,:\, k \in A_i\} \textrm{  and  } 
W_i = \myspan\{e_k \,:\, k \in B_i\}.
\]
Then hypotheses (i), (ii), and (iii) for the set pairs $\{(A_i,B_i)\,:\, i \in [m]\}$  imply hypotheses (i), (ii), and (iii), respectively, of Theorem~\ref{thm:GKsubspaces} for the   
subspace pairs $\{(U_i,W_i)\,:\, i \in [m]\}$.
\end{proof}

We note that in a recent preprint, Yu, Kong, Xi, Zhang, and Ge~\cite{YKXZG} have independently proved the $t=1$ case of Theorem~\ref{cor:GKMNPTXconj}, which is Gerbner et al's Conjecture 2.4~\cite{GKMNPTX}. They also proceed via a subspace generalization. Their argument uses F\"{u}redi's threshold version of the Two Families Theorem~\cite{Furedi} and the characterization of self-annihilating subspaces (Theorem~\ref{prop:extEKRspaces}) given in the first preprint version of this paper.

%
%
%

\section{\textcolor{black}{Limiting Examples and Questions}}\label{sec:badexamples}
\textcolor{black}{Are our new Two Families theorems optimal?}  It is not clear that we can hope to further relax condition~\ref{it:order} of Theorem~\ref{thm:subspaces} and Corollary~\ref{cor:combo}, which requires that
\[
a_1 \leq a_2 \leq \dots \leq a_m \textrm{ and } b_1 \geq b_2 \geq \dots \geq b_m.
\]
Examples~\ref{ex:death} and~\ref{ex:log} both violate condition~\ref{it:order} for many values of $i$, and both examples satisfy
\[
\sum_{i=1}^m \frac{1}{\binom{n_i}{a_i}} = \Omega(\log m).
\]

However, there are examples achieving a weighted sum greater than 1 that violate condition~\ref{it:order} for just one value of $i$. 
\begin{example}

For $a,b,c >0$, set $n=a+b$.  Build a pair of families by first listing all  profile-$(a,b)$ complementary pairs of subsets of $[n].$ Choose $S \in \sbsts{n}{b+1}$. Any such $S$ intersects non-trivially with each $A_i$ so far. Now append a pair $(A^*,B^*)$ to the list, where $|A^*|=a$, $|B^*|=b+c$, and $S \subseteq B^*$ (the elements of $A^*$ and $B^*$ can otherwise be chosen arbitrarily). The weighted sum is 
\[
\frac{\binom{n}{a}}{\binom{n}{a}} + \frac{1}{\binom{n+c}{a}}>1. 
\] 

\end{example}

\begin{example} Fix $a,b,c,d > 0$ and $a>c$. Let $n=a+b$. Build a system by first listing all  profile-$(a,b)$ complementary pairs of subsets of $[n],$ then all  profile-$(a-c,b+c+d)$ complementary pairs of subsets of $[n+d]$. This pair of families is skew cross-intersecting; note that when $d=0$ it is two ``levels'' of Example~\ref{ex:death}. However, the weighted sum is
\[
\frac{\binom{n}{a}}{\binom{n}{a}} + \frac{\binom{n+d}{a-c}}{\binom{n+d}{a-c}} =2.
\]
\end{example}

\medskip

What about other relaxations of the cross-intersecting condition? For example, would it be enough to require the full cross-intersecting condition for pairs with distinct profiles, but only skew for pairs with the same profile? 

\begin{conjecture}
 Let $(A_1,B_1),\dots,(A_m,B_m)$ 
be pairs of finite non-empty subsets of $\N$. Write $a_i=|A_i|$ and $b_i=|B_i|$ for $1\leq i\leq m$.  
Suppose that
\begin{enumerate}[label=\textup{(\roman*)}]
\item $A_i \cap B_i = \emptyset$ for $1 \leq i \leq m$, 
\item $A_i \cap B_j \not= \emptyset$ for $1 \leq i < j \leq m$, and
\item $A_i \cap B_j \not= \emptyset$ if $|A_i|\ne |A_j|$ or $|B_i|\ne |B_j|$.
\end{enumerate}
Then 
\begin{equation}\label{eq:all2}
\sum_{i=1}^m \frac{1}{\binom{a_i+b_i}{a_i}} \leq 1.
\end{equation}
\end{conjecture}

Finally, we note that several other directions of generalization have been studied.  For example, Tuza~\cite{Tuza87} further weakened the skew condition~\ref{it:skewintersect} to require only that at least one of $A_i\cap B_j$ and $A_j \cap B_i$ be non-trivial for each $1 \leq i,j, \leq m$, $i\not=j$, a version considered further by {Kir\'{a}ly},
Nagy, {P\'{a}lv\"{o}lgyi}, and Visontai~\cite{KNPV}. F\"{u}redi~\cite{Furedi}, Talbot~\cite{Talbot}, and Kang, Kim, and Kim~\cite{KangKimKim},  considered \textcolor{black}{stronger intersection conditions}, while Einstein~\cite{Einstein} (corrected in Oum and Wee~\cite{OW}) and O'Neill and Verstraete~\cite{OV} look at more than two families of sets. Can any of these variations be further addressed with exterior algebra methods?



\end{document}